\documentclass[10pt, reqno]{amsart}
\usepackage{graphicx} 
\usepackage{amsmath,amssymb,amsthm,mathtools}
\usepackage[a4paper,margin=3cm]{geometry}
\usepackage{amsfonts}
\usepackage{enumerate}
\usepackage[inline]{enumitem}
\numberwithin{equation}{section}
\usepackage{amsaddr}
\usepackage{bbm}
\usepackage{xcolor}
\usepackage{array}
\usepackage{rotating}
\usepackage{xcolor}
\usepackage[colorlinks=true,
            linkcolor=blue,
            citecolor=red,
            urlcolor=blue]{hyperref}
\usepackage{amsmath,amssymb,mathtools}
\usepackage{graphicx}

\theoremstyle{definition}
\newtheorem{definition}{Definition}[section]
\theoremstyle{plain}
\newtheorem{theorem}[definition]{Theorem}
\newtheorem{proposition}[definition]{Proposition}
\newtheorem{lemma}[definition]{Lemma}
\newtheorem{corollary}[definition]{Corollary}
\theoremstyle{remark}
\newtheorem{remark}[definition]{Remark}

\newcommand{\R}{\mathbb{R}}

\newcommand{\N}{\mathbb{N}}
\newcommand{\E}{\mathbb{E}}

\newcommand{\probmr}{\mathcal{M}(\R)}
\newcommand{\probmsymr}{\mathcal{M}^{sym}(\R)}

\begin{document}

\title{Functional inequalities for Boolean entropy}
\author{Guillaume Cébron, Kewei Pan}
\address[A1,A2]{Univ Toulouse, INSA Toulouse, CNRS, IMT, Toulouse, France.}
\email{guillaume.cebron@math.univ-toulouse.fr, kewei.pan@math.univ-toulouse.fr}

\begin{abstract}
Building on the recently introduced notion of Boolean entropy, we define the corresponding Boolean Fisher information via a de Bruijn identity. We study the monotonicity of this Fisher information in the Boolean Central Limit Theorem and establish several functional inequalities involving these quantities, including a logarithmic Sobolev inequality. We also develop Non-microstate counterparts and prove the associated functional inequalities. In addition, we introduce a notion of Stein discrepancy in the Boolean setting, which leads to new Berry--Esseen type bounds in the Boolean central limit theorem.

\end{abstract}

\maketitle

\tableofcontents

\section*{Introduction}
\noindent \textbf{Boolean entropy and Fisher information.}
In~\cite{pan2025entropy}, the second author introduced a notion of \emph{Boolean entropy} analogous to Voiculescu's free entropy. This Boolean entropy arises as the rate function in large deviation principles for the empirical spectral distribution of certain random matrix models. More precisely, two models were investigated: the GUE conditioned to have a majority of zero eigenvalues, and the GUE conditioned to have a majority of zero entries. For both models, a large deviation principle is established for the empirical distribution of the nonzero eigenvalues. The corresponding rate functions coincide (up to a multiplicative constant) and lead to the following definition. For any probability measure $\mu$ on the real line with finite second moment, the \textit{Boolean entropy} of $\mu$ is defined by the logarithmic integral
$$\Gamma(\mu):=\int \log |x|\,d\mu(x)\in [-\infty,+\infty).$$
The terminology Boolean entropy reflects the fact that the two random matrix models arise from earlier works \cite{cebron2024asymptotic,lenczewski2014limit} in which asymptotic Boolean independence (as defined in~\cite{speicher1996boolean}) naturally emerges, just as free entropy is defined by Voiculescu from the GUE model \cite{arous1997large,voiculescu1993analogues}, which is an instance of asymptotic free independence. Boolean independence forms the third component of the triptych — \emph{classical statistical independence / free independence / Boolean independence}~— put forward by Speicher~\cite{speicher1996universal} as the only three universal independences in noncommutative probability. The definition of Boolean entropy is thus the first stone of the development of a full theory for Boolean entropy, paralleling the classical and the free settings. In~\cite{pan2025entropy}, it is proved that the Boolean entropy satisfies two properties analogous to those of classical and free entropy: it remains monotone along the Boolean central limit theorem; and it is maximized by the limiting distribution, namely the Rademacher distribution $\mathbb{\mathrm{b}}=\frac{1}{2}\delta_{-1}+\frac{1}{2}\delta_{+1}$ (the symmetric Bernoulli distribution).

In this article, we introduce the \emph{Boolean Fisher information} $\Psi(\mu)$ by mimicking the de Bruijn identity: starting from a symmetric measure, and differentiating the Boolean entropy along the Boolean convolution semigroup generated by $\mathbb{\mathrm{b}}$, we obtain
$$\Psi(\mu)=\iint \frac{\log x^2-\log y^2}{x^2-y^2}\,d\mu(x)d\mu(y).$$
Starting from a non-symmetric measure leads to a different expression for the associated Fisher information, which appears significantly more difficult to analyze. For this reason, we systematically work with symmetric probability measures in the rest of this paper and leave the investigation of the non-symmetric Fisher information for future work.

We prove that the Boolean entropy and the Boolean Fisher information satisfy the Boolean analogues of the Shannon-Stam inequality, the Blachman-Stam inequality, the Cramér-Rao inequality, the entropy-power inequality and the Voiculescu-Stam inequality, which concern sums of Boolean independent random variables. Moreover, we establish the monotonicity of Boolean Fisher information, which provides another perspective of proving the monotonicity of Boolean entropy. These results show that these quantities are well adapted to Boolean independence.

$ $

\noindent \textbf{Functional inequalities.} Then, we consider $\Gamma(\mu)$ and $\Psi(\mu)$ with respect to the equilibrium measure $\mathbb{\mathrm{b}}$. The \emph{Boolean entropy relative to $\mathbb{\mathrm{b}}$} is defined, for any symmetric probability measure $\mu$, by
$$\Gamma(\mu|\mathbb{\mathrm{b}}):=\frac{1}{2}\int x^2\,d\mu(x) -\int \log |x|\,d\mu(x)-\frac{1}{2}\in [0,+\infty].$$
This is in fact the full rate function mentioned earlier — when the potential term is included — and which is minimized at $\Gamma(\mathbb{\mathrm{b}}|\mathbb{\mathrm{b}})=0$ as we will see (see also~\cite{pan2025entropy}). We define also the \emph{Boolean Fisher information $\Psi(\mu|\mathbb{\mathrm{b}})$ relative to $\mathbb{\mathrm{b}}$}  as the derivative of $\Gamma(\mu|\mathbb{\mathrm{b}})$ along the Boolean Ornstein-Uhlenbeck process, which admits the following explicit representation: for any symmetric probability measure $\mu$ on~$\mathbb{R}$,
$$\Psi(\mu|\mathbb{\mathrm{b}})=\iint \frac{\log x^2-\log y^2}{x^2-y^2}\,d\mu(x)d\mu(y) + \int x^2\,d\mu(x)-2\in [0,+\infty].$$
The positivity of the Boolean Fisher information is not trivial, and implies the monotonicity of the Boolean entropy along the Boolean Ornstein-Uhlenbeck process.

Following Voiculescu, the two quantities introduced above can be qualified as \textit{Microstates}, since they arise from random matrix approximations.   
The Non-microstates approach to free entropy, introduced by Voiculescu~\cite{voiculescu1998analogues}, defines free entropy without relying on Microstates. In the Boolean case, this intrinsic approach allows us to introduce an alternative version of the Fisher information. The \emph{Non-microstates Boolean Fisher information relative to $\mathbb{\mathrm{b}}$} is defined, for any symmetric probability measure~$\mu$ on $\mathbb{R}$, by
$$\Psi^*(\mu|\mathbb{\mathrm{b}}):=\int \frac{1}{x^2}\,d\mu(x)+ \int x^2\,d\mu(x)-2\in [0,+\infty].$$
Integrating this quantity along the Boolean Ornstein-Uhlenbeck process yields the \emph{Non-microstates Boolean entropy relative to $\mathbb{\mathrm{b}}$}, given, for any symmetric probability measure $\mu$ on $\mathbb{R}$, by
\begin{equation*}
    \Gamma^*(\mu|\mathbb{\mathrm{b}})=\frac{1}{2}\log\left(\int \frac{1}{x^2}\,d\mu(x)\right)+\frac{1}{2}\int x^2\,d\mu(x)-\frac{1}{2}\in [0,+\infty].
\end{equation*}
Finally, in parallel with Stein's method developed 
in~\cite{ledoux2015stein,fathi2017free} for the classical and free settings respectively, we define the \textit{Boolean Stein discrepancy} $D^*(\mu|\mathbb{\mathrm{b}})$. It is given, for any symmetric probability measure $\mu$ on $\mathbb{R}$, by
\begin{equation*}
    D^*(\mu|\mathbb{\mathrm{b}}):=\left(\int x^4\,d\mu(x)-2\int x^2\,d\mu(x)+1\right)^{1/2}\in [0,+\infty].
\end{equation*}

We establish the following functional inequalities which provide a clearer understanding of how these quantities are related, and in particular how they all provide upper bounds on the quadratic Wasserstein-2 distance $W_2$ from the Rademacher distribution $\mathbb{\mathrm{b}}$. The striking parallel between these inequalities and their classical and free counterparts provides yet another argument for considering these quantities as the proper Boolean analogues of the classical and free ones.
For any symmetric probability measure $\mu$ on $\mathbb{R}$, we have:
\begin{figure}[h]
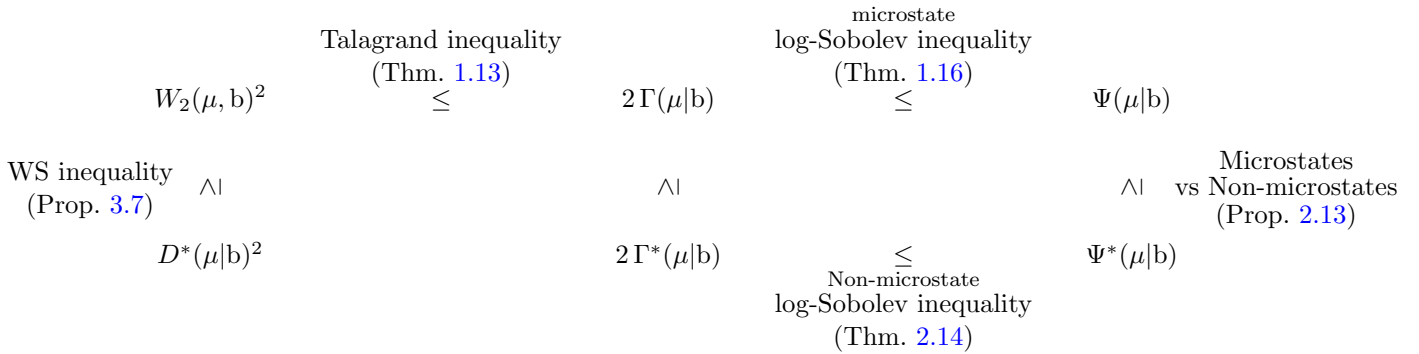

\centering
\setlength\arraycolsep{1.8em}
{$
\begin{array}{@{} c @{\qquad} c @{\qquad} c @{\qquad} c @{\qquad} c @{}}

W_2(\mu,\mathbb{\mathrm{b}})^2
& \overset{\vcenter{\hbox{\shortstack{
Talagrand inequality\\
(Thm.~\ref{thm:TTI})
}}}}{\leq}
& 2\,\Gamma(\mu|\mathbb{\mathrm{b}})
& \overset{\vcenter{\hbox{\shortstack{\footnotesize
microstate\\
log-Sobolev inequality\\
(Thm.~\ref{thm:LSI})
}}}}{\leq}
& \Psi(\mu|\mathbb{\mathrm{b}}) \\[1.3em]

\mathllap{\vcenter{\hbox{\shortstack{
WS inequality\\
(Prop.~\ref{prop:WS})
}}}\quad}
\rotatebox{90}{$\geq$}
&&
\rotatebox{90}{$\geq$}
&&
\mathrlap{\hspace{0.7cm}\vcenter{\hbox{\shortstack{
Microstates\\
vs Non-microstates\\
(Prop.~\ref{thm:micro-nonmicro})
}}}}
\rotatebox{90}{$\geq$}
\\[1.3em]

D^*(\mu|\mathbb{\mathrm{b}})^2
&&
2\,\Gamma^*(\mu|\mathbb{\mathrm{b}})
& \underset{\vcenter{\hbox{\shortstack{\footnotesize
Non-microstate\\
log-Sobolev inequality\\
(Thm.~\ref{thm:nmLSI})
}}}}{\leq}
&
\Psi^*(\mu|\mathbb{\mathrm{b}})

\end{array}
$}
\caption{Hierarchy of Boolean functional inequalities.}
\label{fig:boolean-functional-inequalities}
\end{figure}

Moreover, we are able to prove two refined versions of the diagonal log-Sobolev inequality $ 2\Gamma(\mu|\mathbb{\mathrm{b}})\leq \Psi^*(\mu|\mathbb{\mathrm{b}})$: the direct Boolean analogue of the Otto–Villani HWI inequality, namely the \emph{Boolean HWI inequality} (Thm.~\ref{thm:HWI_sym}) 
$$2\Gamma(\mu|\mathbb{\mathrm{b}})
\leq
2 W_2(\mu,\mathbb{\mathrm{b}})\,\sqrt{\Psi^*(\mu|\mathbb{\mathrm{b}})}
-W_2(\mu,\mathbb{\mathrm{b}})^2\leq \Psi^*(\mu|\mathbb{\mathrm{b}})$$
and the direct Boolean analogue of the Ledoux-Nourdin-Peccati HSI inequality, namely the \emph{Boolean HSI inequality} (Thm.~\ref{thm:HSI}) 
$$ 2\Gamma(\mu|\mathbb{\mathrm{b}})\leq D^*(\mu|\mathbb{\mathrm{b}})^2\cdot \log\left(1+\frac{\Psi^*(\mu|\mathbb{\mathrm{b}})}{D^*(\mu|\mathbb{\mathrm{b}})^2}\right)\leq \Psi^*(\mu|\mathbb{\mathrm{b}}).$$
Unlike the one-dimensional free case,  we observe a difference between the Microstates entropy and the Non-microstates one.

Starting from a not necessarily symmetric measure $\mu$, all the quantities involved but $W_2(\mu,\mathbb{\mathrm{b}})^2$ are invariant under taking the symmetrization $\mu^{\mathrm{s}}(B):=\frac{1}{2}\mu(B\cup (-B))$ of $\mu$. Consequently, replacing $W_2(\mu,\mathbb{\mathrm{b}})^2$ by $W_2(\mu^{\mathrm{s}},\mathbb{\mathrm{b}})^2=\int||x|-1|^2 d\mu(x)$, the eight inequalities above remain true for all probability measures on $\mathbb{R}$, not necessarily symmetric.

$ $

\noindent \textbf{Berry-Esseen inequalities.} The WS inequality provides a direct control of the Wasserstein distance and yields a new quantitative estimate in the Boolean central limit theorem for the Boolean convolution $\uplus$. For any measure $\mu$, we denote by $m_k(\mu)$ (with $k\in \mathbb{Z}$) the $k$-th moment $\int x^k d\mu(x)$ of $\mu$ whenever it exists. Let $\mu$ be a probability measure such that the first four moments exist with $m_1(\mu) = 0$, $m_2(\mu) = 1$ ($\mu$ is centered and normalized),
and $m_4(\mu) < +\infty$.  We define $\mu_n$ to be the dilation of $\mu^{\uplus n}$ by $\frac{1}{\sqrt{n}}$, i.e. $\mu_n(B)=\mu^{\uplus n}(\{x:\frac{1}{\sqrt{n}}x\in B\})$. We prove that, whenever $\mu$ is symmetric, for all $n\geq 1$, we have
$$W_2(\mu_n, \mathbb{\mathrm{b}})\leq D^*(\mu_n|\mathbb{\mathrm{b}})=\frac{1}{\sqrt{n}}\sqrt{m_4(\mu)-1}.$$
This follows directly from the transport control provided by the WS inequality.
The rate of convergence $O(n^{-1/2})$ is the usual one in Berry-Esseen inequalities for the classical central limit theorem~\cite{berry1941accuracy,esseen1945fourier} and for the free one~\cite{chistyakov2008limit,kargin2007berry}. In the Boolean case, to the best of the authors' knowledge, there are only four existing Berry-Esseen type results: the rate of convergence $O(n^{-1/3})$ was obtained for the (weaker) Lévy distance under the same assumptions~\cite{arizmendi2018berry}, the rate of convergence $O(n^{-1/2})$ was obtained for the Lévy distance under the (stronger) assumption of existence of the $6$-th moment~\cite{salazar2022berry}, the rate of convergence $O(n^{-1/2})$ was obtained for an \textit{ad-hoc} distance under the (weaker) assumption of existence of the $3$-th moment~\cite{hamdan2023fixed}, and the rate of convergence $O(n^{-1/2})$ was obtained for a coupling distance for measures with bounded support~\cite{jekel2019operad}. The assumption of a \textit{symmetric} measure is not needed in the four works cited, and since the Stein discrepancy directly controls the Wasserstein-$1$ distance $W_1$ without symmetry assumptions we are able to remove this assumption as well~: 
for all $n\geq 1$, we have
$$W_1(\mu_n, \mathbb{\mathrm{b}})\leq D^*(\mu_n|\mathbb{\mathrm{b}})=\frac{1}{\sqrt{n}}\sqrt{m_4(\mu)-1}.$$
Beyond metric convergence, one may also study information-theoretic convergence. In the classical (and also free) central limit theorem, it is also possible to quantify the convergence of the Fisher information and the convergence of the relative entropy.  The optimal rate is $O(n^{-1})$ under the appropriate conditions~\cite{artstein2004rate,chistyakov2013asymptotic,johnson2004fisher}. However, in the free and classical case, the Stein method allows to reach quite directly the (almost optimal) rate $O(\log(n)/n)$ under very weak conditions~\cite{ledoux2015stein,fathi2017free}. The Boolean Stein's method  developed in this paper allows to get the following quantitative entropic central limit theorems. 
Let $\mu$ be a symmetric probability measure with the existing moments $m_1(\mu) = 0$, $m_2(\mu) = 1$, $m_4(\mu) <+\infty$ and $m_{-2}(\mu) < +\infty$. Then, as $n\to \infty$,

$$ \Gamma(\mu_n|\mathbb{\mathrm{b}})=O\left(\frac{\log(n)}{n}\right)\ \ \ \text{ and }\ \ \ \Psi(\mu_n|\mathbb{\mathrm{b}})=O\left(\frac{1}{\sqrt{n}}\right).$$

Notice that, under the same conditions, $\Gamma^*(\mu_n|\mathbb{\mathrm{b}})$ and $\Psi^*(\mu_n|\mathbb{\mathrm{b}})$  remain constant as $n$ increases, which reveals a structural difference in behavior between the Microstates and Non-microstates quantities.

\subsection*{Organisation of the paper}

In Section~\ref{Sec1:booleanentropy}, we introduce the Boolean entropy $\Gamma(\mu)$ and the Boolean Fisher information $\Psi(\mu)$, establish the Boolean de Bruijn identity along the Boolean heat semigroup, and prove the monotonicity of $\Gamma(\mu)$ and $\Psi(\mu)$ in the Boolean central limit theorem. 
In Section~\ref{Sec2:booleanentropy}, we develop the main information-theoretic properties of Boolean entropy and Fisher information, including the Shannon--Stam inequality, the Blachman--Stam inequality, the Cramér--Rao inequality, the entropy power inequality, and the Voiculescu--Stam inequality. 
In Section~\ref{Sec3:booleanentropy}, we introduce the relative Boolean entropy $\Gamma(\mu|\mathbb{\mathrm{b}})$ and relative Boolean Fisher information $\Psi(\mu|\mathbb{\mathrm{b}})$, and we prove the Boolean Talagrand transport inequality, the Boolean log-Sobolev inequality, and the exponential decay of relative entropy along the Boolean Ornstein--Uhlenbeck process.

In Section~\ref{sec:Schwinger}, we introduce the Boolean Schwinger--Dyson equation which is a motivation for subsequent definitions. In Section~\ref{Sec1:booleanentropy*}, we define the Non-microstates Boolean Fisher information $\Psi^*(\mu)$, the Non-microstates Boolean entropy $\Gamma^*(\mu)$ and establish the corresponding de Bruijn identity. In Section~\ref{Sec2:booleanentropy*}, we prove the main properties of the Non-microstates quantities, including the Blachman--Stam inequality, the Cramér--Rao inequality, the Voiculescu--Stam equality, the entropy power equality, and the constancy in the Central Limit Theorem. In Section~\ref{sec:nmfunctional_inequalities}, we introduce the Non-microstates relative Boolean entropy $\Gamma^*(\mu|\mathbb{\mathrm{b}})$ and Non-microstates relative Boolean Fisher information $\Psi^*(\mu|\mathbb{\mathrm{b}})$, we prove the Non-microstate log-Sobolev  inequality, establish the comparison between the Microstates and Non-microstates quantities and prove   HWI inequality.

In Section~\ref{sec:Steinmethod}, we adapt Stein’s method to the Rademacher target $\mathbb{\mathrm{b}}$, and derive a Berry--Esseen bound for $W_1$.
In Section~\ref{sec:Stein_discrepancy}, we introduce the Boolean Stein discrepancy $D^*(\mu|\mathbb{\mathrm{b}})$ and relate it to transport, proving in particular the WS inequality and the explicit Berry--Esseen control of $D^*(\mu_n|\mathbb{\mathrm{b}})$ along the Boolean CLT.
In Section~\ref{sec:EntropicCLT}, we establish the HSI inequality and deduce quantitative entropic convergence in the Boolean CLT and the WSH inequality.

\subsection*{Notations} Let us denote by 
$\mathcal{M}(\R)$ the set of probability measures on $\mathbb{R}$ and by $\mathcal{M}^{sym}(\R)$ the subset of symmetric probability measures on $\mathbb{R}$. The Cauchy transform of $\mu \in \mathcal{M}(\mathbb{R})$ is defined by
$$
G_{\mu}(z)=\int \frac{1}{z-x}\,d\mu(x),\quad z \in \mathbb{H}_+ := \{ z \in \mathbb{C} : \Im z > 0 \}.
$$
Writing $z = x + i\epsilon$ with $\epsilon > 0$, we have the decomposition $G_\mu(x+i\varepsilon) = \pi( Q_\epsilon \ast \mu(x) - i P_\epsilon\ast \mu(x)),$
where
$$
P_\varepsilon(t) = \frac{1}{\pi}\frac{\epsilon}{t^2+\epsilon^2} ,
\quad
Q_\epsilon(t) = \frac{1}{\pi}\frac{t}{t^2+\epsilon^2}
$$
are respectively the Poisson kernel and the \emph{conjugate} Poisson kernel. Recall that $P_\epsilon\ast \mu$ weakly converges to $\mu$ as $\epsilon\to 0^+$. Moreover, for any $f \in L^1_{\mathrm{loc}}(\mathbb{R})$, we have $P_\epsilon \ast f(x) \longrightarrow f(x)$ for every Lebesgue point of $f$, and the convergence is uniform if $f\in C_0(\R)$. The self-energy (or reciprocal Cauchy transform, also called the $K$-transform) of $\mu$ is given by $K_\mu(z):=z-1/G_{\mu}(z)$  for $z\in\mathbb{H}_+$. For $\mu,\nu\in \mathcal{M}(\R)$, the Boolean convolution  introduced by Speicher and Woroudi \cite{speicher1996boolean} is the unique measure $\mu\uplus \nu\in \mathcal{M}(\R)$ such that
\begin{equation*}
	K_{\mu\uplus \nu}(z)=K_{\mu}(z)+K_{\nu}(z),\ \forall z\in\mathbb{H}_+.
\end{equation*}
For $\mu\in \mathcal{M}(\R)$, we denote by $m_k(\mu)$ (with $k\in \mathbb{Z}$) the $k$-th moment $\int x^k d\mu(x)$ of $\mu$ whenever it exists. For $k\geq 1$, the Wasserstein-$k$ distance (or Kantorovich–Rubinstein metric) between $\mu,\nu\in \mathcal{M}(\R)$ is defined by
   \begin{equation*}
		W_k(\mu,\nu):=\inf_{\pi} \left(\iint |x-y|^k\,d\pi(x,y)\right)^{\frac{1}{k}} ,
	\end{equation*}
   where the infimum is over all couplings of $(\mu,\nu)$. We will not use the Lévy distance $d_{\text{Lévy}}$ in this article, but if one wants to transfer our result to the Lévy distance, let us mention that
   $$\sqrt{d_{\text{Lévy}}(\mu,\nu)} \leq W_1(\mu,\nu)\leq W_2(\mu,\nu)\ \ \text{ and }\ \ 
 d_{\text{Lévy}}(\mu,\nu)\leq W_2(\mu,\nu).$$
 If $X$ and $Y$ are random variables with respective law $\mu$ and $\nu$, we set
 $W_1(X,Y):=W_1(\mu,\nu)$, and the Kantorovich duality yield
 $$W_1(X,Y)=\sup_{\varphi\ 1-\text{Lipschitz}}|\mathbb{E}[\varphi(X)]-\mathbb{E}[\varphi(Y)]|.$$
We will consider Boolean independent variables. Boolean independence is a non-unital
phenomenon, and cannot in general be realized in a $W^*$-probability space. In this article, a \textit{non-commutative probability space} $(\mathcal{M},\tau)$ is a von Neumann algebra $\mathcal{M}$ with a nondegenerate normal state $\tau$ which is not necessarily faithful. Elements of $\mathcal{M}$ are called \textit{(non-commutative random) variables}. For any self-adjoint $A=A^*\in \mathcal{M}$, the law $\mathcal{L}(A)$ of $A$ is the unique measure $\mu\in \mathcal{M}(\R)$ such that, for all $k\in \mathbb{N}$, we have
$$\tau(A^k)=m_k(\mu).$$
The Cauchy transform $G_{\mu}$ will be often denoted by $G_A$.

Two self-adjoint variables $X,Y$ are \textit{Boolean independent} if and only if, for all finite sequences of integers $k_1,k_2, \ldots \geq 1$, we have
$$\tau(X^{k_1}Y^{k_2}X^{k_3}\cdots)=\tau(X^{k_1})\tau(Y^{k_2})\tau(X^{k_3})\cdots $$
and
$$\tau(Y^{k_1}X^{k_2}Y^{k_3}\cdots)=\tau(Y^{k_1})\tau(X^{k_2})\tau(Y^{k_3})\cdots,$$
and in this case, the law of the sum is given by the Boolean convolution: $$\mathcal{L}(X+Y)=\mathcal{L}(X)\uplus\mathcal{L}(Y).$$
In the symmetric case, i.e. whenever $\mathcal{L}(X),\mathcal{L}(Y)\in \mathcal{M}^{sym}(\R)$, all odd moments of $X$ and $Y$ vanish. In particular, $\mathcal{L}(X)\uplus\mathcal{L}(Y)$ is also symmetric. Moreover, we get
\begin{equation}\label{eq:square}
    \mathcal{L}((X+Y)^2)=\mathcal{L}(X^2+Y^2)
\end{equation}
which will be useful several times in this article. More generally, for all $\mu, \nu\in \mathcal{M}^{sym}(\R)$, if we denote by $\mu^{(2)}, \nu^{(2)}$ and $(\mu\uplus\nu)^{(2)}$ the push-forward of $\mu, \nu$ and $\mu\uplus\nu$ by the map $x\mapsto x^2$, we have
$$(\mu\uplus\nu)^{(2)}=\mu^{(2)}\uplus \nu^{(2)} $$
as well-documented in~\cite{mlotkowski2011symmetrization}.

For a probability measure $\mu$ on $\mathbb{R}$ with finite fourth moment, we may also
define the first four Boolean cumulants $r_1(\mu)$, $r_2(\mu)$, $r_3(\mu)$, and $r_4(\mu)$ of $\mu$ by
the equations (see~\cite{arizmendi2018berry,speicher1996boolean})
\begin{align*}
m_1(\mu) &= r_1(\mu), \\
m_2(\mu) &= r_1(\mu)^2 + r_2(\mu), \\
m_3(\mu) &= r_1(\mu)^3 + 2 r_1(\mu) r_2(\mu) + r_3(\mu), \\
m_4(\mu) &= r_1(\mu)^4 + 3 r_1(\mu)^2 r_2(\mu) + r_2(\mu)^2 + 2 r_3(\mu) r_1(\mu) + r_4(\mu).
\end{align*}
Similarly to the classical cumulants, the Boolean cumulants defined above
satisfy for $\mu, \nu \in \mathcal{M}(\mathbb{R})$ and $i \in \mathbb{N}$ that
\begin{equation}
r_i(\mu \uplus \nu) = r_i(\mu) + r_i(\nu)\notag
\end{equation}
and if $\nu$ is the dilation of $\mu$ by $a$, we have
\[
r_i(\nu) = a^i r_i(\mu).\]
Let $\mu$ be a probability measure such that the first four moments exist with $m_1(\mu) = 0$, $m_2(\mu) = 1$,
and $m_4(\mu) < +\infty$.  We define $\mu_n$ to be the dilation of $\mu^{\uplus n}$ by $\frac{1}{\sqrt{n}}$, i.e. $\mu_n(B)=\mu^{\uplus n}(\{x:\frac{1}{\sqrt{n}}x\in B\})$. Using the two previous rules, we get
\begin{align}
m_1(\mu_n)=r_1(\mu_n) &=m_1(\mu)= 0, \notag\\
m_2(\mu_n)=r_2(\mu_n) &=m_2(\mu)= 1,\notag\\
m_4(\mu_n)-1=r_4(\mu_n) &= \frac{1}{n}r_4(\mu)=\frac{1}{n}(m_4(\mu)-1).\label{eq:fourthcumulant}
\end{align}

$ $

\noindent \textbf{Acknowledgement.} We warmly thank Octavio Arizmendi and Djalil Chafaï for insightful discussions related to the topics of this paper.

\setcounter{tocdepth}{2}

\section{The Microstates approach}\label{Sec:booleanentropy}

\subsection{Boolean entropy and Fisher information}\label{Sec1:booleanentropy}
	In this section, we introduce a notion of Boolean Fisher information and show that it interacts with Boolean entropy via an integral formula that mirrors the classical de Bruijn identity. This provides a differential characterization of Boolean entropy along a Boolean analogue of a heat flow. Let us first recall the definition of Boolean entropy.
\begin{definition}
let $X$ be a self-adjoint operator of a non-commutative probability space $(\mathcal{M},\tau)$ with law $\mu\in\mathcal{M}(\R)$. The \textit{Boolean entropy} $\Gamma(X)$ of $X$ is defined as
\begin{equation}\label{def:entropy}
		\Gamma(X):=\int \log |x|\,d\mu(x)\in [-\infty,+\infty).
	\end{equation}
    More generally, for any $\mu \in\mathcal{M}(\R)$, we define the \textit{Boolean entropy} $\Gamma(\mu)$ of $\mu$ as
\begin{equation*}
		\Gamma(\mu):=\int \log |x|\,d\mu(x) = \frac{1}{2}\int \log x\,d\mu^{(2)}(x)\in [-\infty,+\infty)
	\end{equation*} provided that $\int_1^{+\infty}\log x\,d\mu^{(2)}(x)<+\infty$. Recall that $\mu^{(2)}\in \mathcal{M}(\R_+)$ is the push forward measure of $\mu$ by the map $x\mapsto x^2$
\end{definition}

Throughout this section, unless explicitly stated otherwise, all probability measures are assumed to be symmetric. In the Boolean case, the heat semigroup is given by the Boolean convolution with a scaled Rademacher variable: starting with a self-adjoint operator $X$ of a non-commutative probability space $(\mathcal{M},\tau)$ whose law is symmetric, we consider the variable $X+\sqrt{t}B$ where $B\in \mathcal{M}$ is Boolean independent from $X$ and distributed according to the Rademacher law $\mathrm{b}= \frac{1}{2}\delta_{-1}+\frac{1}{2}\delta_1.$ Our first result, Theorem~\ref{thm:de bruijn}, establishes the fundamental link between the Boolean entropy and the Boolean Fisher information, in the form of a Boolean de Bruijn identity. It expresses the increment of Boolean entropy along the Boolean convolution semigroup $(\mathcal{L}(X+\sqrt{t}B))_{t\geq 0}$ generated by the Rademacher variable in terms of a quantity that we introduce now as the Boolean Fisher information.

\begin{definition}
Let $X$ be a self-adjoint operator of a non-commutative probability space $(\mathcal{M},\tau)$ with law $\mu\in\mathcal{M}^{sym}(\R)$. The \textit{Boolean Fisher information} $\Psi(X)$ of $X$ is defined as
\begin{equation}\label{def:fisher}
	\Psi(X):=\iint \frac{\log x^2-\log y^2}{x^2-y^2} \,d\mu(x)d\mu(y)\in (0,+\infty]
\end{equation}
understood with the continuous extension on the diagonal.
More generally, for any $\mu \in\mathcal{M}^{sym}(\R)$, we define the \textit{Boolean Fisher information} $\Psi(\mu)$ of $\mu$ as
\begin{equation*}
    \Psi(\mu):=\iint \frac{\log x^2-\log y^2}{x^2-y^2} \,d\mu(x)d\mu(y) = \iint \frac{\log x-\log y}{x-y} \,d\mu^{(2)}(x)d\mu^{(2)}(y)\in (0,+\infty].
\end{equation*}
\end{definition}
    Classically, Fisher information arises as the derivative of entropy along the Brownian motion, a relation known as the de Bruijn identity. Voiculescu~\cite{voiculescu1993analogues} extended this framework to free probability, introducing the notions of free entropy and free Fisher information. In the Boolean setting, an analogous de Bruijn identity can also be established.
	\begin{theorem}\label{thm:de bruijn}
    Let $X$ be a self-adjoint operator of a non-commutative probability space $(\mathcal{M},\tau)$, with law $\mu\in\mathcal{M}^{sym}(\R)$,  and $B\in \mathcal{M}$ be Boolean independent from $X$ and distributed according to~$\mathrm{b}$. Fix $t\geq 0$ such that $\Gamma(X+\sqrt{s}B)<\infty,\ \Psi(X+\sqrt{s}B)<+\infty$ for any $s\in[0,t].$ Then we have
	\begin{equation}\label{eq:de bruijn}
		\Gamma(X+\sqrt{t}B) - \Gamma(X) = \int_0^t \frac{1}{2}\Psi(X + \sqrt{s}B )\,ds.
	\end{equation}
    In particular, $\Gamma(X+\sqrt{t}B)$ is increasing with respect to $t$.
	\end{theorem}
    It is worth noting that the Fisher information $\Psi$ solely depends on the push forward probability measure $\mu^{(2)}\in\mathcal{M}(\R_+)$ by $x\mapsto x^2$. Therefore, without altering the expression of $\Psi$, the Fisher information can be defined on $\probmr$. However, the de Bruijn identity generally fails under this naïve extension. Indeed, if the law of $X$ is not symmetric, we have another explicit expression for the Fisher information satisfying the above de Bruijn relation, given by
	\begin{equation*}
		\Psi(X)=\Psi(\mu):=\iint \frac{\frac{\log x^2}{x}-\frac{\log y^2}{y}}{x-y}\,d\mu(x)d\mu(y)\in [-\infty,+\infty],\quad \text{if}\ X\sim \mu\in \probmr.
	\end{equation*}
	We emphasize that unlike the classical and free Fisher information, this quantity is not necessarily nonnegative and the translation invariance, scaling property (i.e.  $\Psi(\lambda X+m)=\lambda^{-2}\Psi(X)$ for any $\lambda, m\in \R\setminus\{0\}$) are no longer satisfied. Nevertheless, in the case where $\mu$ is symmetric, it can be shown to reduce exactly to the Fisher information defined in \eqref{def:fisher}. Moreover, the symmetric framework provides substantially better analytical properties, under which many results from the classical theory, as well as their free analogues, can be recovered. For this reason, we work from now on in the symmetric framework unless explicitly stated otherwise.
	\begin{proof}
Since 
$X$ and $B$ are self-adjoint elements of a von Neumann algebra, they are bounded operators. Hence the spectral measure of $X$ and of $X+\sqrt{t}B$ are compactly supported. Thanks to \eqref{eq:square}, we know that $\mathcal{L}((X+\sqrt{t}B)^2)=\mathcal{L}(X^2+t B^2)$, and we can compute
\begin{align*}
    G_{(X+\sqrt{t}B)^2}(z) &=G_{X^2+tB^2}(z)\\
    &=\frac{1}{z-K_{X^2+tB^2}(z)}\\
    &=\frac{1}{z-K_{X^2}-K_{tB^2}(z)}\\
    &=\frac{1}{1/G_{X^2}(z)-K_{tB^2}(z)}.
\end{align*}
Now,  since $\mathcal{L}(t B^2)=\delta_t$,  we have
$$ K_{tB^2}(z)=z-\frac{1}{G_{tB^2}(z)}
    =z-\frac{1}{(z-t)^{-1}}
    =t,$$
which yields
\begin{equation}G_{(X+\sqrt{t}B)^2}(z)=\frac{1}{1/G_{X^2}(z)-t}=\frac{G_{X^2}(z)}{1-tG_{X^2}(z)}.	\label{eq:Cauchytimet}\end{equation}
In particular, we have
\begin{equation*}
	\frac{\partial }{\partial t}G_{(X+\sqrt{t}B)^2}(z) = (G_{(X+\sqrt{t}B)^2 }(z))^2.
\end{equation*}
To abbreviate the notation, we denote by $\mu(t,\cdot)=\mathcal{L}(X+\sqrt{t}B)$ and $G(t,z)=G_{(X+\sqrt{t}B)^2}(z)$ so that $G(t,z)$ is the Cauchy transform of $\mu^{(2)}(t)$. Put $u(t,z)=\pi^{-1}\Re G(t,z),\ v(t,z)=-\pi^{-1}\Im G(t,z)$, By the above relation, for $z=x+i\epsilon$ we have
\begin{equation*}
	\frac{\partial }{\partial t}v(t,z) = -\frac{1}{\pi}\Im (G(t,z))^2 = 2\pi u(t,z)v(t,z).
\end{equation*}
Denote by $f_\delta(x)=\log |x+i\delta|$ for any $\delta>0$ and $$\Gamma_\delta(\mu)=\int_{\R_+} f_\delta(x)\,d\mu(x),\quad \Psi_\delta(\mu)= \iint_{\R_+^2} \frac{f_\delta(x)-f_\delta(y)}{x-y}\,d\mu(x)d\mu(y).$$
For any $0\le s\le t$ and fixed $\epsilon>0$, since $\mu^{(2)}(s)$ is compactly supported, we have $|\frac{\partial}{\partial s}v(s,x+i\epsilon)|\leq |(G(s,x+i\epsilon))^2|\leq \frac{C\epsilon}{1+|x|^2}$. Thus,
\begin{align*}
	\frac{d}{ds}\Gamma_{\delta}(v(s,\cdot+i\epsilon)) &= \int f_\delta(x)\cdot \frac{\partial}{\partial s}v(s,x+i\epsilon)\,dx\\
    &= -\frac{1}{\pi}\Im \int f_\delta(x)\cdot (G(s,x+i\epsilon))^2\,dx\\
	&= -\frac{1}{\pi}\Im \int f_\delta(x)\int \frac{1}{x+i\epsilon-y}\,d\mu^{(2)}(s,y)\int \frac{1}{x+i\epsilon-y'}\,d\mu^{(2)}(s,y')\,dx\\
	&= -\frac{1}{\pi}\Im \iint \frac{1}{y-y'}\left[\int \left(\frac{f_\delta(x)}{x+i\epsilon-y} - \frac{f_\delta(x)}{x+i\epsilon-y'} \right)\,dx\right]\,d\mu^{(2)}(s,y)d\mu^{(2)}(s,y').\\
    &=\iint \frac{1}{y-y'}( P_\epsilon\ast f_\delta (y) - P_\epsilon\ast f_\delta (y'))\,d\mu^{(2)}(s,y)d\mu^{(2)}(s,y').
\end{align*}
Moreover, note that $$ \Gamma_\delta(v(t,\cdot+i\epsilon)) = \int \log |x+i\delta|\cdot v(t,x+i\epsilon)\, dx = \int \log |x+i(\delta+\epsilon)|\,d\mu(t,x) = \Gamma_{\delta+\epsilon}(\mu(t)),$$ hence, we have
\begin{equation}
    \Gamma_{\delta+\epsilon}(\mu^{(2)}(t))-\Gamma_{\delta+\epsilon}(\mu^{(2)}) = \int_0^t \iint \frac{P_\epsilon\ast f_\delta (y) - P_\epsilon\ast f_\delta (y')}{y-y'}\,d\mu^{(2)}(s,y)d\mu^{(2)}(s,y')\,ds,
\end{equation}
Let $\epsilon\to 0^+,$ the integrand on the right hand side converges to $\Psi_\delta(\mu(t))$ due to the fact that $P_\epsilon\ast f_\delta\longrightarrow f_\delta$ uniformly on the support of $\mu^{(2)}(s)$. To conclude, we get $\Gamma_\delta(\mu^{(2)}(t)) - \Gamma_\delta(\mu^{(2)}) = \int_0^t \Psi_\delta(\mu^{(2)}(s))\,ds$. Finally let $\delta\to 0^+$, and the monotone convergence theorem yields \eqref{eq:de bruijn}, where the factor $\frac{1}{2}$ comes from the definition \ref{def:entropy}.
\end{proof}

From the definitions of $\Gamma(X)$ and $\Psi(X)$, we have the following scaling property.
\begin{proposition}
    Let $X$ be a self-adjoint operator of a non-commutative probability space $(\mathcal{M},\tau)$ with law $\mu\in\mathcal{M}^{sym}(\R)$, and let $\lambda\in \mathbb{R}\setminus\{0\}$.  
    We have $$\Gamma(\lambda X) = \Gamma(X) + \log |\lambda|;\ \Psi(\lambda X) =\lambda^{-2} \Psi(X).$$
    The scaling property still holds if we extend the definition of $\Psi$ to $\probmr$ without altering the expression. However, the translation invariance of $\Psi$ fails.
\end{proposition}
    Now let us consider the variables $X_t= e^{-t}X+\sqrt{1-e^{-2t}}B$, which has the same distribution as the variables in the Boolean Ornstein-Uhlenbeck process. By the scaling property of $\Gamma$, we know that
	\begin{equation*}
		\Gamma(X_t)=\Gamma\left(X+\sqrt{e^{2t}-1}B\right)-t.
	\end{equation*}
	Thus, with the de Bruijn identity \eqref{eq:de bruijn} and the scaling property of $\Psi$, we have
	\begin{equation}\label{eq:fisher ou}
		\frac{d}{dt}\Gamma(X_t)=e^{2t}\Psi\left(X+\sqrt{e^{2t}-1}B\right)-1= \Psi(X_t)-1.
	\end{equation}

	\begin{corollary}\label{cor:de bruijn_0^infty}
		Suppose that $\Gamma(X+\sqrt{s}B)<\infty$ and $ \Psi(X+\sqrt{s}B)<+\infty$ for all $s\in[0,+\infty)$. Then we have
		\begin{equation}\label{eq:de bruijn_0^infty}
			\Gamma(X)=\frac{1}{2} \int_0^\infty \left( \frac{1}{s+1}-\Psi(X+\sqrt{s}B) \right)\,ds
		\end{equation}
	\end{corollary}
	\begin{proof}
		By the previous computation, and clearly $\lim_{t\to 0}\Gamma(X_t)=\Gamma(X),\ \lim_{t\to+\infty}\Gamma(X_t)=\Gamma(B)$, thus we have
		\begin{align*}
			\Gamma(B)-\Gamma(X) &= \int_0^{+\infty} \left[e^{2t}\Psi\left(X+\sqrt{e^{2t}-1}B\right)-1 \right]\,dt\\
			&= \int_0^{+\infty} \left[ (s+1)\Psi(X+\sqrt{s}B)-1 \right]\frac{1}{2(s+1)}\,ds,
		\end{align*}
		where the second equality is derived from change of variable $s\mapsto \sqrt{e^{2t}-1}$. Therefore, with the fact that $\Gamma(B)=0$, we complete the proof.
	\end{proof}

    The monotonicity of classical entropy originates in the work of Shannon (see \cite{shannon1948}, \cite{stam1959}, \cite{lieb1978}) and was settled in \cite{artstein2004}. Adapting a similar strategy, Shlyakhtenko established the free counterpart (see \cite{shlyakhtenko2007}). This result was later revisited by Dadoun and Youssef (see \cite{dadoun2021}) through the maximal correlation method. We emphasize that the proofs of all these results are through the monotonicity of Fisher information, and the key is the $L^2$-structure of the Fisher information (i.e. the conjugate variable). In \cite{pan2025entropy}, we established the Boolean analogue of this monotonicity phenomenon. However, the argument does not involve the Fisher information, but instead relies purely on the analytical structure of the Boolean convolution. For the sake of completeness, we now turn to the monotonicity of the Fisher information itself, but as the quantity $\Psi$ does not yields a nice $L^2$-structure, the way we prove the monotonicity shall be different from the classical and free case.
    \begin{theorem}[Monotonicity of the Fisher information]\label{thm:monofisher}
        Let $\{X_i\}_{i\in \N}\subset (\mathcal{M},\tau)$ be a sequence of Boolean independent and identically distributed self-adjoint random variables with law $\mu\in\probmsymr$.  Define
        \begin{equation*}
            S_n:=\frac{X_1+\cdots+X_n}{\sqrt{n}}.
        \end{equation*}
        Then, for any integers $n\ge m\ge 1$, we have
        \[
        \Psi(S_n)\le \Psi(S_m),
        \]
        with equality if and only if
        \[
        X_1\sim \tfrac{1}{2}\delta_{-a}+\tfrac{1}{2}\delta_a
        \]
        for some $a\neq 0$.
    \end{theorem}    
    \begin{proof}Since the
variables are self-adjoint elements of a von Neumann algebra, they are bounded operators and their spectral measures are compactly supported. 
        We begin with the monotonicity of $\Psi$. Thanks to \eqref{eq:square}, we have 
        \[
        S_n^2 \stackrel{d}{=} \frac{X_1^2+\cdots+X_n^2}{n}.
        \]
        Recall that $\mu_n$ is the dilation of $\mu^{\uplus n}$ by $\frac{1}{\sqrt{n}}$, thus we denote by $\mu^{(2)}_n$ the law of $S_n^2$. Via the change of variable $x\mapsto x^2$, one can reformulate the Fisher information as follows:
        \begin{equation*}
	        \Psi(S_n) = \iint \frac{\log x-\log y}{x-y}\,d\mu^{(2)}_n(x)d\mu^{(2)}_n(y) = \int \frac{1}{x}\left(\int \frac{\log (y/x)}{y/x-1}\,d\mu^{(2)}_n(y)\right)\,d\mu^{(2)}_n(x).
        \end{equation*}
        It therefore suffices to show that for any $X_1\sim\mu\in\probmsymr$ and any $n\ge m$,
        \begin{equation}\label{ineq:monofisher integer}
	          \int \frac{\log x}{x-1}\,d\mu^{(2)}_n(x) \le \int \frac{\log x}{x-1}\,d\mu^{(2)}_m(x).
        \end{equation}
        Indeed, the above inequality (if true for any symmetric law, and in particular for $X_1/\sqrt{x}$ with fixed $x>0$) implies that we have
        $$          \int \frac{\log y/x}{y/x-1}\,d\mu^{(2)}_n(y) \le \int \frac{\log y/x}{y/x-1}\,d\mu^{(2)}_m(y), $$
and allows to write   
        \begin{align*}
            &\Psi(S_n)=\iint \frac{\log x-\log y}{x-y}\,d\mu^{(2)}_n(x)d\mu^{(2)}_n(y) =\int \frac{1}{x}\left(\int \frac{\log (y/x)}{y/x-1}\,d\mu^{(2)}_n(y)\right)\,d\mu^{(2)}_n(x)\\
            &\le \int \frac{1}{x}\left(\int \frac{\log (y/x)}{y/x-1}\,d\mu^{(2)}_m(y)\right)\,d\mu^{(2)}_n(x)=\int \frac{1}{y}\left(\int \frac{\log (x/y)}{x/y-1}\,d\mu^{(2)}_n(x)\right)\,d\mu^{(2)}_m(y)\\
            &\leq \int \frac{1}{y}\left(\int \frac{\log (x/y)}{x/y-1}\,d\mu^{(2)}_m(x)\right)\,d\mu^{(2)}_m(y)
            = \iint \frac{\log x-\log y}{x-y}\,d\mu^{(2)}_m(x)d\mu^{(2)}_m(y) = \Psi(S_m).
        \end{align*}
        
        To establish \eqref{ineq:monofisher integer}, we adopt the same approach as in the proof of the monotonicity of Boolean entropy (Proposition~4.2 in \cite{pan2025entropy}). In view of the definition of $\mu^{(2)}_n$, recall that for any $\alpha>0$, it can be checked that $G_{\alpha X}(z) = \alpha^{-1}G_X(\alpha^{-1}z)$ and $K_{\alpha X}(z)=\alpha K_X(\alpha^{-1}z)$, then by the additivity of the self energy $K$ under Boolean convolution, the Cauchy transform of $\mu^{(2)}_n$ is given by
        \begin{equation*}
            G_{\mu^{(2)}_n}(z)=\frac{1}{z-K_{X_1^2}(nz)}.
        \end{equation*}
        Observe that the discrete family $(\mu^{(2)}_n)_{n\in\mathbb{N}} \subset \mathcal{M}(\R_+)$ admits a natural extension to a continuous family $(\mu^{(2)}_t)_{t\ge 1} \subset \mathcal{M}(\R_+)$, 
defined by

        \begin{equation}\label{def:coninuous extension}
            G_{\mu^{(2)}_t}(z)=\frac{1}{z-K_{X_1^2}(tz)},
        \end{equation}
        and we aim to show a stronger version of the monotonicity:
        \begin{equation}\label{ineq:monofisher real}
            \Psi(\mu^{(2)}_t)\le \Psi(\mu^{(2)}_s),\ \forall t\ge s\ge 1.
        \end{equation}
        In fact, it follows directly from the definition that $\mu^{(2)}_t$ 
        satisfies the multiplicative semigroup property: $(\mu^{(2)}_t)_s = \mu^{(2)}_{ts}$. Consequently, if $\Psi(\mu^{(2)}_t)=+\infty$ for all $t \ge 1$, the inequality is trivial. Otherwise, suppose that there exists $t_0 \in [1,+\infty)$ such that $\Psi(\mu^{(2)}_{t_0}) < +\infty$, while $\Psi(\mu^{(2)}_{s}) = +\infty$ for all $s \in [1,t_0)$. By the semigroup property, we may shift the parameter and assume without loss of generality that $t_0 = 1$, that is, $\Psi(X) < +\infty$. Moreover, to establish the monotonicity for all $t \ge 1$, it suffices to verify the negativity of the derivative with respect to $t$ for an arbitrary $X_1 \sim \mu \in \probmsymr$,
        \begin{equation*}
            \frac{d}{d t} \int \frac{\log x}{x-1}\,d\mu^{(2)}_t(x)\le 0.
        \end{equation*}
        To proceed, let $z=x+i\epsilon$ and write $G_{\mu^{(2)}_t}(z)=\pi\bigl(u_t(z)-iv_t(z)\bigr)$. Taking partial derivative with respect to $t$ on both sides of $\eqref{def:coninuous extension}$ at $t=1$, we obtain
        \begin{equation*}
            \frac{\partial}{\partial t}\Big|_{t=1}G_{\mu^{(2)}_t}(z) = z(G_{\mu^{(2)}_1}(z))^2+zG_{\mu^{(2)}_1}'(z),
        \end{equation*}
        which gives
        \begin{equation*}
            \frac{\partial}{\partial t}\Big|_{t=1}v_t(z) = x[2\pi u_1(z)v_1(z)+v_1'(z)] -\epsilon[u'_1(z)+\pi(u_1^2(z)-v_1^2(z))].
        \end{equation*}
        Let $L(x)=(x\log x)/(x-1)$, it can be checked that $L$ is monotone increasing with $\lim_{x\to 0^+}L(x)=0$ and $\lim_{x\to +\infty}L(x)=+\infty$. Moreover, the singularity of $L$ at $x=1$ can be removed so that $L\in C^1(\R_+)$. In particular, $L$ is bounded in any compact set. For sufficiently large $K>0$, we denote $g_K(x)=\operatorname{sign}(\log x) \cdot K\wedge|\log x|/(x-1)$ and $L_K(x)=xg_K(x)$. Note that $|\frac{\partial}{\partial t}|_{t=1}v_t(z)|\le |z(G_{\mu_1^{(2)}}(z))^2|+|zG'_{\mu_1^{(2)}}(z)|\le \frac{c_\epsilon}{1+|x|}$ because $\mu^{(2)}_1$ is compactly supported, hence we have
        \begin{align*}
            \frac{d}{d t}\Big|_{t=1} \int g_K(x)v_t(x+i\epsilon)\,dx
            &= \int g_K(x) \frac{\partial}{\partial t}\Big|_{t=1}v_t(x+i\epsilon)\,dx \\
            &= \int L_K(x)[2\pi u_1(x+i\epsilon)v_1(x+i\epsilon)+v_1'(x+i\epsilon)]\,dx +o(\epsilon) \\
            &= \iint\left(\frac{P_\epsilon\ast L_K(x)-P_\epsilon\ast L_K(y)}{x-y}\right)\,d\mu^{(2)}_1(x)d\mu^{(2)}_1(y)\\
            &\quad - \int L'_K(x)v_1(x+i\epsilon)\,dx + o(\epsilon).\\
            &:= J_K(v_1(\cdot+i\epsilon))
        \end{align*}
        Therefore, by the multiplicative semigroup property of $\{\mu^{(2)}_t\}_{t\ge 1}$, we have
        \begin{equation*}
            \frac{d}{d t} \int g_K(x)v_t(x+i\epsilon)\,dx = \frac{1}{t} J_K(v_t(\cdot+i\epsilon)),
        \end{equation*}
        and consequently
        \begin{equation*}
            \int g_K(x)v_t(x+i\epsilon)\,dx - \int g_K(x)v_0(x+i\epsilon)\,dx = \int_0^t \frac{1}{s}J_K(v_s(\cdot+i\epsilon))\,ds.
        \end{equation*}
        Note that as $\epsilon\to 0^+$, the weakly convergence of $v_s(\cdot+i\epsilon)$ to $\mu_s$ and the uniform convergence of $P_\epsilon\ast L_K$ to $L_K$ on compact sets yields
        \begin{equation*}
            \int g_K(x)\,d\mu_t(x) - \int g_K(x)\,d\mu_0(x) = \int_0^t \frac{1}{s}\iint\left(\frac{L_K(x)- L_K(y)}{x-y} - L'_K(x)\right)\,d\mu^{(2)}_s(x)d\mu^{(2)}_s(y)\,ds.
        \end{equation*}
        Finally, let $K\to +\infty$ along with the monotone convergence theorem, we conclude that
        \begin{equation*}
            \int \frac{\log x}{x-1}\,d\mu_t(x) - \int \frac{\log x}{x-1}\,d\mu_0(x) = \int_0^t \frac{1}{s}\iint\left(\frac{L(x)- L(y)}{x-y} - L'(x)\right)\,d\mu^{(2)}_s(x)d\mu^{(2)}_s(y)\,ds
        \end{equation*}
        so that
        \begin{equation}
            \frac{d}{d t} \int \frac{\log x}{x-1}\,d\mu^{(2)}_t(x) = \frac{1}{t}\iint\left(\frac{L(x)- L(y)}{x-y} - L'(x)\right)d\mu^{(2)}_t(x)d\mu^{(2)}_t(y).
        \end{equation}
        Since
        \begin{equation*}
            \frac{L(x)-L(y)}{x-y} =\int_0^1 L'(\theta x+(1-\theta)y)\,d\theta.
        \end{equation*}
        The proof is immediate if $L'$ is convex on $(0,\infty)$. Hence, it remains to verify the convexity of $L'$. A direct computation yields
        \begin{equation*}
            L'''(x)=\frac{2\bigl[(x-1)-\log x\bigr]}{x(x-1)^4}\ge 0.
        \end{equation*}
        Since $L'(x)$ admits a continuous extension at $x=1$, we conclude that $L'$ is convex on $(0,+\infty)$. Moreover, the equality holds if and only if $\mu^{(2)}_t$ is a Dirac mass $\delta_a$ for some $a\in(0,+\infty)$. In view of the definition of $\mu^{(2)}_t$, it is clear that if $\mu^{(2)}$ is not a one point Dirac mass, then $\mu^{(2)}_t$ can not be a one point Dirac mass for all $t$. Therefore, $\Psi(S_n)=\Psi(S_m)$ for some $n\ne m\in \N$ if and only if $\mu$ is a scaled Rademacher distribution. This completes the proof of the monotonicity of the Fisher information $\Psi$.
    \end{proof}
    The monotonicity of the entropy $\Gamma$ then follows from Corollary~\ref{cor:de bruijn_0^infty} together with the divisibility of the Rademacher distribution with respect to the Boolean convolution. Therefore, in particular, the entropy $\Gamma$ is non-decreasing along the Boolean central limit theorem. This result was already established in Proposition 4.2 of our first paper \cite{pan2025entropy}, where we do not require $\mu$ to be symmetric. Moreover, it is worth pointing out that, although we only treated the case when $\mu$ is centered and has variance $1$ in \cite{pan2025entropy}, the argument remains valid without these assumptions. Indeed, the centering and variance assumptions are used solely to ensure that $\Gamma(\mu^{(2)}_n)\longrightarrow \Gamma(\mathbb{\mathrm{b}})$. As a consequence, the monotonicity of $\Gamma$ holds in full generality. Thus we state the following theorem and omit the proof. 
    \begin{theorem}[Monotonicity of entropy]
        For a sequence of self-adjoint random variables $\{X_i\}_{i\in \N}$ that are Boolean independent and identically distributed, we have for any integers $n\ge m\ge 1$,
        \begin{equation}
            \Gamma(S_n)\ge \Gamma(S_m).
        \end{equation}
    \end{theorem}

\subsection{Information-theoretic properties}\label{Sec2:booleanentropy}

	The Cramér–Rao inequality and Stam inequality (\cite{stam1959}) are fundamental in classical information theory. It turns out that these results are consistent with Voiculescu's free entropy and free Fisher information (\cite{voiculescu1998analogues}). Motivated by these parallels, we aim to establish corresponding inequalities in the Boolean setting. Before that, we shall point out some properties of the Boolean entropy and Boolean Fisher information, some of these properties will be used in subsequent sections.
	\begin{proposition}\label{prop:infor property}
	 	The Boolean entropy $\Gamma$ and Fisher information $\Psi$ satisfy the following properties  (where we assume that all appearing variables are self-adjoint in some non-commutative probability space and have symmetric laws):
	 	\begin{enumerate}
	 		\item We have the bound for $\Psi$ as follows:
	 		\begin{equation}\label{ineq:fisher bdd}
	 			\tau\left(\frac{2}{X^2+\tilde{X}^2}\right)\le \Psi(X)\le \tau\left( \frac{1}{X^2} \right),
	 		\end{equation}
	 		where $\tilde{X}$ is a classically independent copy of $X$. If $X$ is not invertible, the above formulas should be understood as $\tau\left(\frac{2}{X^2+\tilde{X}^2}\right)=\int 2/(x^2+y^2) d\mu(x)d\mu(y)$ and $\tau\left( \frac{1}{X^2} \right):=\int x^{-2} d\mu(x)$ with $\mu=\mathcal{L}(X)$.
	 		\item $\Gamma$ is upper semicontinuous with respect to the weak topology of probability measures with compact support, and $\Psi$ is lower semicontinuous. In particular, let $\{X_k\}_{k\in \N}$ be a sequence of random variables such that $X_k\to X$ in the sense of operator norm, then
            \begin{equation*}
                \Gamma(X)\ge \limsup_{k\to +\infty}\Gamma(X_k);\ \Psi(X)\le \liminf_{k\to +\infty}\Psi(X_k).
            \end{equation*}
	 		\item If $X$ and $Y$ are Boolean independent, then we have the Boolean analogue of Shannon-Stam inequality and Blachman-Stam inequality: that is, for any $\theta\in [0,1]$,
	 		\begin{align}
	 			&\Gamma(\sqrt{\theta}X + \sqrt{1-\theta}Y) \ge \theta\Gamma(X) + (1-\theta)\Gamma(Y)\label{ineq:SSI};\\
	 			&\Psi(\sqrt{\theta}X + \sqrt{1-\theta}Y)\le \theta\Psi(X) + (1-\theta)\Psi(Y)\label{ineq:BSI}.
	 		\end{align}
	 		
	 	\end{enumerate}
	\end{proposition}
	\begin{proof}
 To show the \emph{lower bound} in \eqref{ineq:fisher bdd}, it suffices to prove the following inequality
		\begin{equation}\label{ineq:lower bound}
			\frac{\log t-\log s}{t-s}\geq \frac{2}{t+s}.
		\end{equation}
	Without loss of generality, we may assume that $t\geq s$ so that we can rewrite the inequality into
        \begin{equation*}
			\left( \frac{t}{s}+1 \right)\cdot \log \left( \frac{t}{s} \right)- 2\left( \frac{t}{s}-1 \right)\geq 0.
		\end{equation*}
		Here we denote by $F(\zeta):=(\zeta+1)\log \zeta-2(\zeta-1)\geq 0$ for $\zeta\geq 1.$ Observe that $F(1)=0$ and $F'(\zeta)=\log \zeta+ \frac{1}{\zeta}-1\geq 0,$ we deduce that $F(\zeta)\geq 0$ when $\zeta\geq 1$ as wanted.
        
 For the \emph{upper bound}, we note that
		\begin{equation}\label{eq:fisher_0^1}
			\Psi(X) = \iint \frac{\log x^2-\log y^2}{x^2-y^2}\,d\mu(x)d\mu(y) = \iint \int_0^1 \frac{1}{tx^2 + (1-t)y^2}\,dt\,d\mu(x)d\mu(y),
		\end{equation}
        where $\mu$ is the law of $X$. 
		The convexity of $1/x$ and Jensen's inequality yields the bound.

		To prove $(2)$, assume that $\{(\mu_k)_{k\ge 1},\mu\}$ are compactly supported probability measures such that $\mu_k\to \mu$ in the weak topology. Note that $\Gamma(\mu) = \lim_{\epsilon\to 0^+}\Gamma_\epsilon(\mu)$, where
		\begin{equation*}
			\Gamma_\epsilon(\mu) = \int \log(\vert x \vert +\epsilon)\,d\mu(x).
		\end{equation*}
		Therefore, for any $\epsilon>0$, we have
		\begin{equation*}
			\limsup_{k\to +\infty}\Gamma(\mu_k) \le \limsup_{k\to +\infty} \Gamma_\epsilon(\mu_k) = \Gamma_\epsilon(\mu),
		\end{equation*}
		then let $\epsilon\to 0^+$ on the right hand side yields the upper semicontinuity. Similarly, we can show that $\Psi$ is lower semicontinuous.
        
		For $(3)$, due to the equation \eqref{eq:de bruijn_0^infty} and divisibility of Rademacher distribution under Boolean convolution, it suffices to show the Blachman-Stam inequality \eqref{ineq:BSI}.

    Let us denote $ \mu^{(2)}, \nu^{(2)} $ as the laws of operators $ X^2, Y^2 $, respectively, and let $ m^{(2)}_{\theta} $ represent the law of the operator $ \theta X^2 + (1-\theta) Y^2 $. For the sake of simplifying the notation, we still write $\Psi({\mu^{(2)}}):= \Psi(X)=\iint \frac{\log x -\log y}{x-y}\,d\mu^{(2)}(x)d\mu^{(2)}(y)$ (the same for $\nu^{(2)}$ and $m^{(2)}_\theta$). Thanks to \eqref{eq:square}, the Blachman-Stam inequality becomes:
    \begin{equation*}
        \Psi(m^{(2)}_{\theta})\leq \theta \Psi(\mu^{(2)}) + (1-\theta)\Psi(\nu^{(2)}).
    \end{equation*}
    If $\Psi(X)=+\infty$ or $\Psi(Y)=+\infty$, we automatically have the inequality. Hence, we may assume that $\Psi(X)$ and $\Psi(Y)$ are finite. Moreover, there is nothing to prove when $\theta=0$ or $\theta=1$. To proceed, for $\theta\in (0,1)$, we begin by analyzing the Cauchy transform of $ m^{(2)}_\theta $. By the scaling of the Cauchy transform and the additivity of the self energy $K$ under Boolean convolution, we have,
    \begin{align}\label{eq:antilinearrelation}
G_{m^{(2)}_{\theta}}(z)
&=\frac{1}{z-K_{m^{(2)}_{\theta}}(z)}\\
&=\frac{1}{z-\theta K_{\mu^{(2)}}(\theta^{-1}z)-(1-\theta)K_{\nu^{(2)}}((1-\theta)^{-1}z)}\nonumber\\
&=\frac{1}{\theta\bigl(z-K_{\mu^{(2)}}(\theta^{-1}z)\bigr)+(1-\theta)\bigl(z-K_{\nu^{(2)}}((1-\theta)^{-1}z)\bigr)}\nonumber\\
&=\frac{1}{\theta\,\frac{1}{G_{\mu^{(2)}_{\theta}}(z)}+(1-\theta)\,\frac{1}{G_{\nu^{(2)}_{1-\theta}}(z)}},
\end{align}
    where the Cauchy transform of $\mu^{(2)}_\theta,\nu^{(2)}_{1-\theta}$ are defined via
		\begin{equation*}
			G_{\mu^{(2)}_\theta}(z):= \frac{1}{ z- K_{\mu^{(2)}}(\theta^{-1}z) }; \quad G_{\nu^{(2)}_{1-\theta}}(z):= \frac{1}{ z- K_{\nu^{(2)}}((1-\theta)^{-1}z) }.
		\end{equation*}
		In view of the expression $\eqref{def:coninuous extension}$, thus $\mu^{(2)}_{\theta}$ coincides with the continuous extension of the law of $S_n^2$ evaluated at $n=\theta^{-1}$ (this is well defined since $\theta\in (0,1)$), and the same constructions applies to $\nu^{(2)}_{1-\theta}$. Furthermore, recall that $\mu^{(2)}_\theta,\nu^{(2)}_{1-\theta}$ are supported on the positive real line. Now, the main observation is the following: assume that $x,y>0$, then
        \begin{equation*}
            \frac{\log x - \log y}{x-y} = \int_{0}^{+\infty} \frac{1}{(x+t)(y+t)}\,dt,
        \end{equation*}
        hence, by Fubini's theorem, for any probability measure $\mu$ supported on the positive real line, we have
        \begin{equation*}
            \iint \frac{\log x - \log y}{x-y} \,d\mu(x)d\mu(y) = \int_{0}^{+\infty} \left( \int \frac{1}{x+t}\,d\mu(x) \right)^2\,dt = \int _0^{+\infty} \left(G_{\mu}(-t)\right)^2\,dt.
        \end{equation*}
        Since $\mu$ is supported on $\R_+$, the Cauchy transform $G_{\mu}(-t)$ is well defined for $t>0$ and we can extend the equality \eqref{eq:antilinearrelation} to $z\in \R_-$. Therefore, using the Jensen's inequality twice, we get
        \begin{align*}
            \Psi(m^{(2)}_{\theta}) = \int_0^{+\infty} (G_{m^{(2)}_{\theta}}(-t))^2\,dt &= \int_0^{+\infty}\left( \frac{1}{\theta\frac{1}{G_{\mu^{(2)}_{\theta}}(-t)}+(1-\theta)\frac{1}{G_{\nu^{(2)}_{1-\theta}}(-t)}} \right)^2\,dt\\
            &\leq  \int_0^{+\infty} \left( \theta G_{\mu^{(2)}_\theta}(-t) + (1-\theta) G_{\nu^{(2)}_{1-\theta}}(-t) \right)^2\,dt\\
            &\leq  \int_0^{+\infty} \left[ \theta (G_{\mu^{(2)}_\theta}(-t))^2 + (1-\theta) (G_{\nu^{(2)}_{1-\theta}}(-t))^2\right]\,dt\\
            &= \theta\Psi(\mu^{(2)}_{\theta}) + (1-\theta)\Psi(\nu^{(2)}_{1-\theta})\\
            &\le \theta \Psi(\mu^{(2)}) + (1-\theta)\Psi(\nu^{(2)}).
        \end{align*}
        In the last inequality, we use the monotonicity of the Fisher information $\Psi$, see \eqref{ineq:monofisher real}.
	\end{proof}
    As a direct consequence of this proposition, the Boolean counterparts of the information inequalities can be easily deduced.
	\begin{corollary}[Cramér-Rao inequality]
    Let $X\in(\mathcal{M},\tau)$ be a self-adjoint random variable with a symmetric law. We have
	\begin{equation}
		\Psi(X)\ge \frac{1}{\tau(X^2)}.
	\end{equation}
	\end{corollary}
	\begin{proof}
		It suffices to show that $\Psi(X)\tau(X^2)\ge 1$. By the lower bound of $\Psi$ in \eqref{ineq:fisher bdd}, along with H\"{o}lder inequality, we have
		\begin{equation*}
			\Psi(X)\tau(X^2)\ge \tau\left(\frac{2}{X^2+\tilde{X}^2}\right)\cdot \tau\left( \frac{X^2+\tilde{X}^2}{2} \right)\ge 1.
		\end{equation*}
	\end{proof}

    \begin{corollary}[Boolean Stam inequality]\label{cor:booleanstam}
	Let $X,Y$ be self-adjoint random variables of a non-commutative probability space $(\mathcal{M},\tau)$ with symmetric laws and assume that $X,Y$ are Boolean independent. We have
	\begin{equation}\label{ineq:booleanstam}
		\frac{1}{\Psi(X+Y)}\ge \frac{1}{\Psi(X)} + \frac{1}{\Psi(Y)}.
	\end{equation}
    In particular, if $\Psi(X+Y)=+\infty$, then $\Psi(X)=\Psi(Y)=+\infty$.
	\end{corollary}
	\begin{proof}
		First we assume that $\Psi(X),\Psi(Y)$ are finite, and by the Blachman-Stam inequality \eqref{ineq:BSI} and the scaling property of Fisher information $\Psi$, for any $\theta\in [0,1]$, we have
		\begin{align*}
			\Psi(X+Y) &= \Psi\left(\sqrt{\theta}\frac{X}{\sqrt{\theta}} + \sqrt{1-\theta}\frac{Y}{\sqrt{1-\theta}} \right)\\
			&\le \theta \Psi\left( \frac{X}{\sqrt{\theta}} \right) + (1-\theta)\Psi\left( \frac{Y}{\sqrt{1-\theta}} \right)\\
			&= \theta^2\Psi(X) + (1-\theta)^2\Psi(Y).
		\end{align*}
        Let $F(\theta)= \theta^2\Psi(X) + (1-\theta)^2\Psi(Y) $, since the above inequality is valid for all $\theta\in [0,1]$, we deduce that 
		\begin{equation*}
			\Psi(X+Y)\le \inf_{\theta\in [0,1]}F(\theta).
		\end{equation*}
		Set $F'(\theta) = 0$ gives a unique solution $\theta_0 = \Psi(Y)/(\Psi(X)+\Psi(Y))\in [0,1] $ and it is easy to see that this is a minimum point. Finally, one can check that
		\begin{equation*}
			F(\theta_0) = \frac{\Psi(X)\Psi(Y)}{\Psi(X)+\Psi(Y)},
		\end{equation*}
        which gives the desired inequality.
        
		If $\Psi(X)=\Psi(Y)=+\infty$, the equality is obvious. Else if either $\Psi(X)=+\infty$ or $\Psi(Y)=+\infty$, we may assume without loss of generality that $\Psi(X)<+\infty$ and $\Psi(Y)=+\infty$, then \eqref{ineq:booleanstam} reduces to $\Psi(X+Y)\le \Psi(X)$. Let $\mu,\nu\in \probmsymr$ be the law of $X,Y$ respectively, note that there exists a sequence of compactly supported symmetric probability measures $\{\nu_n\}_{n\in \N}$ such that $\Psi(\nu_n)<+\infty$ and $\nu_n$ weakly converge to $\nu$. Since Boolean convolution preserve the weak convergence i.e. $\mu\uplus \nu_n\longrightarrow \mu\uplus \nu$ (because the Cauchy transform converges). Thus, by the lower semicontinuity of $\Psi$, we have
        \begin{equation}
            \Psi(X+Y)=\Psi(\mu\uplus\nu)\le \liminf_{n\to+\infty}\Psi(\mu\uplus\nu_n)\le \Psi(\mu)=\Psi(X),
        \end{equation}
		the last inequality holds because $\Psi(\nu_n)<+\infty$ and goes back to the finite case with $\theta=1$, hence we complete the proof.
	\end{proof}
    
	\begin{corollary}[Entropy power inequality]\label{cor:ent power ineq}
	Let $X,Y$ be self-adjoint random variables of a non-commutative probability space $(\mathcal{M},\tau)$ with symmetric laws and assume that $X,Y$ are Boolean independent. We have
	\begin{equation}\label{ineq:ent power ineq}
	 	\exp(2\Gamma(X+Y))\ge \exp(2\Gamma(X)) + \exp(2\Gamma(Y)).
	\end{equation}
    In particular, if $\Gamma(X+Y)=-\infty$, then $\Gamma(X)=\Gamma(Y)=-\infty$.
	\end{corollary}
	\begin{proof}
	    If $\Gamma(X)=\Gamma(Y)=-\infty$, then the inequality is trivial, hence we may assume without loss of generality that $\Gamma(X)>-\infty$ and $\Gamma(Y)=-\infty$. Following the same argument as the preceding proof for Stam inequality, it can also be checked that $\exp{2\Gamma(X+Y)}\ge \exp{2\Gamma(X)}$. Now if $\Gamma(X),\Gamma(Y)$ are finite, by the Shannon-Stam inequality \eqref{ineq:SSI}, and the scaling property of entropy $\Gamma,$ for any $\theta\in [0,1]$, we have
        \begin{align*}
			\Gamma(X+Y) &= \Gamma\left(\sqrt{\theta}\frac{X}{\sqrt{\theta}} + \sqrt{1-\theta}\frac{Y}{\sqrt{1-\theta}} \right)\\
			&\ge \theta \Gamma\left( \frac{X}{\sqrt{\theta}} \right) + (1-\theta)\Gamma\left( \frac{Y}{\sqrt{1-\theta}} \right)\\
			&= \theta \left(\Gamma(X) - \frac{1}{2}\log \theta\right) + (1-\theta)\left(\Gamma(Y) - \frac{1}{2}\log(1-\theta)\right).
		\end{align*}
        Now taking $\exp2(\cdot)$ on both sides gives
        \begin{equation*}
            \exp(2\Gamma(X+Y))\ge \left( \frac{\exp(2\Gamma(X))}{\theta}\right)^{\theta}\cdot \left( \frac{\exp(2\Gamma(Y))}{1-\theta}\right)^{1-\theta}.
        \end{equation*}
        Notice that the choice of  $$\theta=\frac{e^{2\Gamma(X)}}{e^{2\Gamma(X)} + e^{2\Gamma(Y)}}$$ yields the result.
	\end{proof}
    As observed from the proof, the entropy power inequality \eqref{ineq:ent power ineq} (Stam inequality \eqref{ineq:booleanstam}, respectively) is equivalent to Shannon-Stam inequality \eqref{ineq:SSI} (Blachman-Stam inequality \eqref{ineq:BSI}, respectively).

\subsection{Functional inequalities}\label{Sec3:booleanentropy}
\begin{definition}
let $X$ be a self-adjoint operator of a non-commutative probability space $(\mathcal{M},\tau)$ with law $\mu\in\probmr$. The \textit{Boolean entropy $\Gamma(X|\mathbb{\mathrm{b}})$ of $X$ relative to $\mathbb{\mathrm{b}}$}  is defined as
\begin{equation}\label{def:rela entropy}
		\Gamma(X|\mathbb{\mathrm{b}}):=\frac{1}{2}\int x^2 \,d\mu(x)  -\int \log |x|\,d\mu(x)-\frac{1}{2}\in [0,+\infty].
	\end{equation}
    More generally, for any $\mu \in\probmr$, we define the \textit{Boolean entropy $\Gamma(\mu|\mathbb{\mathrm{b}})$ of $\mu$ relative to $\mathbb{\mathrm{b}}$}   as
\begin{equation*}
		\Gamma(\mu|\mathbb{\mathrm{b}}):=\frac{1}{2}\int x^2 \,d\mu(x)  -\int \log |x|\,d\mu(x)-\frac{1}{2}\in [0,+\infty].
	\end{equation*}
    The positivity of $\Gamma(\mu|\mathbb{\mathrm{b}})$ is immediate from the definition.
\end{definition}
    Transportation cost inequalities play an important role in optimal transport, which estimate the square of the 2-Wasserstein distance $W_2$ and relative entropy between a pair of probability measures. The classical one was first obtained by Talagrand \cite{talagrand1996transportation} in 1996. Later on, the free analogue in one dimension was introduced by Biane and Voiculescu \cite{bianevoi2001free}, and it was extended to multivariate case by Hiai and Ueda \cite{hiai2006free}. In the following, we present the Boolean version of such inequalities in one dimension.
	\begin{theorem}[Boolean Talagrand Transport Inequality]\label{thm:TTI}
		For any $\mu\in\mathcal{M}^{sym}(\R)$, we have
		\begin{equation}\label{ineq:TTI}
		    W_2^2(\mu,\mathbb{\mathrm{b}})\leq 2\Gamma(\mu|\mathbb{\mathrm{b}}).
	    \end{equation}
	\end{theorem}
	\begin{proof}
    Since the target measure has two atoms, optimal couplings reduce to mass splitting: any joint law $\pi$ on $\mathbb{R}^2$ with marginal $\mu$ and $\mathbb{\mathrm{b}}$ can be decomposed into $\mu_+\otimes \delta_1 + \mu_-\otimes \delta_{-1},$ where $\mu_-,\mu_+$ are positive measures such that $\mu_-+\mu_+=\mu$ and $\mu_-(\mathbb{R})=\mu_+(\mathbb{R})=\frac{1}{2}$. In particular, the optimal coupling is when $\mu_+$ is supported on $[0,+\infty)$ and $\mu_-$ is supported on $(-\infty,0]$, then
	     \begin{equation*}
		    W_2^2(\mu,\mathbb{\mathrm{b}}) =\int_{\R_+} |x-1|^2\,d\mu_+(x) + \int_{\R_-} |x+1|^2\,d\mu_-(x) = \int (|x|-1)^2\,d\mu(x).
	    \end{equation*}
        Thus, using the inequality that $-2|x| \leq -2\log |x| -2$ gives
	     \begin{align*}
		    W_2^2(\mu,\mathbb{\mathrm{b}})&= \int (|x|-1)^2\,d\mu(x)\\
		    &\leq \int (x^2 - 2\log |x| -1)\,d\mu(x) = 2\Gamma(\mu|\mathbb{\mathrm{b}}).
	    \end{align*}
	\end{proof}
    Let us come back to the Ornstein-Uhlenbeck process. Starting from a self-adjoint operator $X$ of a non-commutative probability space $(\mathcal{M},\tau)$ whose law is symmetric, we consider the variable $X_t:=e^{-t}X + \sqrt{1-e^{-2t}}B$, where $B\in \mathcal{M}$ is Boolean independent from $X$ and distributed according to the Rademacher law $\mathrm{b}.$ We denote by $\mu_{X_t}$ the law of $X_t$. We have
    $$\tau(X_t^2)=e^{-2t}\tau(X^2)+1-e^{-2t}.$$
We define our relative Boolean Fisher information via taking derivative with respect to $t$ of $-\Gamma(X_t|B)$. Using the equation \eqref{eq:fisher ou} we get
	\begin{align}\label{eq:rela fisher}
		-\frac{d}{dt}\Gamma(X_t|\mathbb{\mathrm{b}})\nonumber
		&=\frac{d}{dt}\left(\Gamma(X_t)-\frac{1}{2}\tau\left(X_t^2\right)\right)\\
        &=\iint \frac{\log x^2-\log y^2}{x^2-y^2}\,d\mu_{X_t}(x)d\mu_{X_t}(y) -1 + e^{-2t}\int x^2 \,d\mu_{X_0}(x) - e^{-2t}\nonumber \\
		&= \Psi(X_t) + \tau(X_t^2) - 2
	\end{align}
    which motivates the following definition.
    \begin{definition}
let $X$ be a self-adjoint operator of a non-commutative probability space $(\mathcal{M},\tau)$ with law $\mu\in\mathcal{M}^{sym}(\R)$. The \textit{Boolean Fisher information} $\Psi(X| \mathbb{\mathrm{b}})$ of $X$ relative to $B$ is defined as
\begin{equation}\label{def:rela fisher}
		\Psi(X|\mathbb{\mathrm{b}}):=\iint \frac{\log x^2-\log y^2}{x^2-y^2}\,d\mu(x)d\mu(y)+\int x^2 \,d\mu(x)  -2\in [0,+\infty].
	\end{equation}
    More generally, for any $\mu \in\mathcal{M}^{sym}(\R)$, we define the \textit{Boolean Fisher information $\Psi(\mu|\mathbb{\mathrm{b}})$ of $\mu$ relative to $\mathbb{\mathrm{b}}$}   as
\begin{equation*}
		\Psi(\mu|\mathbb{\mathrm{b}}):=\iint \frac{\log x^2-\log y^2}{x^2-y^2}\,d\mu(x)d\mu(y)+\int x^2 \,d\mu(x)  -2\in [0,+\infty].
	\end{equation*}
\end{definition}
	A priori, it is not obvious that this expression should be nonnegative. Nevertheless, the following result shows that Boolean Fisher information retains the fundamental positivity property known from classical and free settings.

	\begin{proposition}
		The relative Boolean Fisher information $\Psi(\mu| \mathbb{\mathrm{b}})$ is non-negative for all $\mu\in \mathcal{M}^{sym}(\R)$ and $\Psi(\mu| \mathbb{\mathrm{b}})=0$ if and only if $\mu=\mathbb{\mathrm{b}}$.	
	\end{proposition}
	\begin{proof}
		From the explicit expression above,
		\begin{equation}
			\Psi(\mu|\mathbb{\mathrm{b}})= \iint \left(\frac{\log x^2-\log y^2}{x^2-y^2} + \frac{1}{2}x^2 + \frac{1}{2}y^2 - 2\right)\,d\mu(x)d\mu(y).
		\end{equation}
		To prove that $\Psi(\mu|\mathbb{\mathrm{b}})\geq 0$ for all symmetric probability measures $\mu,$ we intend to show that the integrant is almost everywhere non-negative. By the inequality \eqref{ineq:lower bound}, together with the inequality $a^2+b^2\geq 2ab$ yield
		\begin{equation*}
			\Psi(\mu|\mathbb{\mathrm{b}})\geq \iint \left( \frac{2}{x^2+y^2} + \frac{1}{2}(x^2+y^2) - 2 \right)\,d\mu(x)d\mu(y) \geq 0.
		\end{equation*}
		This completes the first part of the proof. Based on the previous argument, it is clear that the equality holds if and only if $\mu=\mathbb{\mathrm{b}}.$
	\end{proof}

	The log-Sobolev inequality (LSI) is a fundamental tool in analysis and probability to study convergence to equilibrium, concentration phenomenon, namely the relative entropy is bounded by the relative Fisher information. We aim to establish the Boolean counterpart, which mirrors the celebrated classical log-Sobolev inequality of Gross \cite{gross1975logarithmic}, and the free analogue was first discovered by Voiculescu \cite{voiculescu1998analogues} and recovered by Biane and Speicher \cite{biane2001free}.
	\begin{theorem}[LSI]\label{thm:LSI}
		For any $\mu\in\mathcal{M}^{sym}(\R)$ with finite variance, we have
		\begin{equation}\label{ineq:LSI}
		   \Gamma(\mu|\mathbb{\mathrm{b}})\leq \frac{1}{2}\Psi(\mu| \mathbb{\mathrm{b}}).
	    \end{equation}
	\end{theorem}
	\begin{proof}
    If $\Psi(\mu|\mathbb{\mathrm{b}})=+\infty$, we automatically have the inequality. Now assume that $\Psi(\mu|\mathbb{\mathrm{b}})<+\infty$, and since $\mu$ has finite variance, in view of the expression of $\Gamma(\mu|\mathbb{\mathrm{b}})$ and $\Psi(\mu|\mathbb{\mathrm{b}})$, we can reduce the inequality to the following:
	\begin{equation*}
		\int \log|x|\,d\mu(x)+\frac{1}{2}\left( \iint \frac{\log x^2-\log y^2}{x^2-y^2}\,d\mu(x)d\mu(y)-1 \right)\geq 0.
	\end{equation*}
		As the first step, observe that the map $x\mapsto x^2$ gives a bijection between symmetric probability measures on $\R$ and probability measures on $\R_+$. Therefore, the above is equivalent to showing that for any $\mu^{(2)}\in\mathcal{M}(\R_+)$,
		\begin{equation*}
			\int \log x\,d\mu^{(2)}(x)+\iint \frac{\log x-\log y}{x-y}\,d\mu^{(2)}(x)d\mu^{(2)}(y)-1 \geq 0.
		\end{equation*}
		We may rewrite the left-hand side as the following functional:
		\begin{equation*}
			F(X):=\E[\log X]+\E\left[\frac{\log X -\log \tilde{X}}{X-\tilde{X}}\right]-1,
		\end{equation*}
		where $X,\tilde{X}$ are classical i.i.d random variables with law $\mu^{(2)}$. For any $\lambda>0$, consider the dilation function
		\begin{equation*}
			f_{X}(\lambda):=F(\lambda X)=\log \lambda+ \E[\log X]+ \frac{1}{\lambda} \E\left[\frac{\log X-\log \tilde{X}}{X-\tilde{X}}\right]-1.
		\end{equation*}
		To ensure $f_X(\lambda)\ge 0$ for all $X>0$, it suffices to check that the global minimum of $f_X(\cdot)$ is nonnegative. Differentiating with respect to $\lambda$ and setting the derivative equal to zero gives  $\lambda= \E\left[\frac{\log X-\log \tilde{X}}{X-\tilde{X}}\right]$, which is the unique minimizer of $f_X$. Evaluating $f_X(\cdot)$ at this point, thus it reduces to proving that for any positive random variable $X\sim \mu^{(2)}$ and its classical i.i.d copy $\tilde{X}$,
        \begin{equation}\label{ineq:equi LSI}
		 	\E[\log X]+\log\left\{ \E\left[ \frac{\log X-\log \tilde{X}}{X-\tilde{X}} \right] \right\}\geq 0.
		\end{equation}
		Using the lower bound from \eqref{ineq:fisher bdd}, it is enough to show that
	    \begin{equation*}
			\E[\log X]+\log\left\{ \E\left[ \frac{2}{X+\tilde{X}} \right] \right\}\geq 0,
		\end{equation*}
	    To proceed, assume $\mu^{(2)} $ is discrete of the form $\mu^{(2)} =\frac{1}{n}\sum_{i=1}^n\delta_{x_i},$ where $x_i>0$ fixed for all $i=1,\ldots, n.$ We aim to show that the above inequality holds for such discrete $\mu^{(2)}$. Plugging it into the left hand side of the above inequality gives
	    \begin{align*}
			LHS &= \frac{1}{n}\log (x_1\cdots x_n) + \log\left(\frac{1}{n^2} \sum_{i=1}^n  \frac{1}{x_i} + \frac{2}{n^2} \sum_{1\leq i<j\leq n}  \frac{2}{x_i+x_j} \right)\\
			&=\log\left\{ \sqrt[n]{x_1\cdots x_n}\left( \frac{1}{n^2} \sum_{i=1}^n  \frac{1}{x_i} + \frac{2}{n^2} \sum_{1\leq i<j\leq n}  \frac{2}{x_i+x_j}  \right) \right\}.
		\end{align*}
	    Hence it remains to show that for all $x_i>0,\ i=1,\ldots,n,$
	    \begin{equation*}
			\sqrt[n]{x_1\cdots x_n}\left( \frac{1}{n^2} \sum_{i=1}^n  \frac{1}{x_i} + \frac{2}{n^2} \sum_{1\leq i<j\leq n}   \frac{2}{x_i+x_j}  \right) \geq 1.
		\end{equation*}
	    By homogeneity, we may assume $x_1x_2\cdots x_n=1$. Under this assumption, the inequality follows from the lemma stated below, which we shall prove later.
	    \begin{lemma}\label{lem:discrete}
			Fix $x_i>0,\ i=1,\ldots,n,$ such that $x_1x_2\cdots x_n=1.$ Then for any integer $n\geq 1,$ we have
			\begin{equation*}
			   \frac{1}{n^2} \sum_{i=1}^n  \frac{1}{x_i} + \frac{2}{n^2} \sum_{1\leq i<j\leq n}  \frac{2}{x_i+x_j} \geq 1,
		\end{equation*}
	    and the equality holds if and only if $x_1=x_2=\cdots =x_n=1.$
		\end{lemma}
		With the above lemma, we have shown that the inequality \eqref{ineq:equi LSI} is valid for all discrete probability measures in the set $\mathcal{D}:=\{\frac{1}{n}\sum_{i=1}^n\delta_{x_i}\ \vert\ x_i>0,\ \forall i=1,\ldots,n,\ n\geq 1\}.$  Since $\mathcal{D}\subset \mathcal{M}(\R_+)$ is dense with respect to the weak topology, using the dominated convergence theorem, we deduce that the inequality holds for all probability measures with support $[\epsilon, \epsilon^{-1}]$ for any $0<\epsilon\leq 1.$ Furthermore, for a general probability measure $\nu $ supported on $\R_+,$ we define
		\begin{equation*}
			\nu_{\epsilon}(A):=\frac{1}{\nu([\epsilon,\epsilon^{-1}])}\nu(A\cap [\epsilon,\epsilon^{-1}])\ \text{for any}\ A\in\mathcal{B}(\R_+)
		\end{equation*}
	   as the truncated measure of $\nu $. Clearly, we have $\nu ([\epsilon,\epsilon^{-1}]) \to 1$ as $\epsilon\to 0.$ In addition, the fact that the function $\log (\cdot)$ is monotone on $\R_+$ implies that $(\log x-\log y)/(x-y)$ is non-negative. Thanks to the monotone convergence theorem, we deduce that
	   \begin{align*}
	   	  \E[\log X\cdot \mathbf{1}_{[\epsilon,\epsilon^{-1}]}]&\overset{\epsilon\to 0}{\longrightarrow} \E[\log X],\\
	   	  \E\left[\frac{\log X-\log \tilde{X}}{X-\tilde{X}}\cdot \mathbf{1}_{[\epsilon,\epsilon^{-1}]^2}\right]&\overset{\epsilon\to 0}{\longrightarrow} \E\left[ \frac{\log X-\log \tilde{X}}{X-\tilde{X}} \right].
	   \end{align*}
	   Therefore, these yield
	   \begin{align*}
	   	  \int \log x\,d\nu_{\epsilon}(x)&\overset{\epsilon\to 0}{\longrightarrow} \int \log x\,d\nu(x),\\
	   	  \iint \frac{\log x-\log y}{x-y}\,d\nu_{\epsilon}(x)d\nu_{\epsilon}(y)&\overset{\epsilon\to 0}{\longrightarrow} \iint \frac{\log x-\log y}{x-y}\,d\nu(x)d\nu(y),
	   \end{align*}
	    which completes the proof.
	\end{proof}
	\begin{proof}(Sketch of proof of Lemma~\ref{lem:discrete})
		This is a minimization problem with constraints. Therefore, we just apply Lagrange multiplier method and we define the following function in the domain $\R_+^n\times \R:$
		\begin{equation*}
			L(x_1,\ldots,x_n,\lambda):= \sum_{i=1}^n  \frac{1}{x_i} + 2\cdot \sum_{1\leq i<j\leq n}  \frac{2}{x_i+x_j} + \lambda(x_1x_2\cdots x_n-1).
		\end{equation*}
		Setting all partial derivatives equal to zero yields the system:
		\begin{align*}
			&\frac{\partial L}{\partial x_i}= -\frac{1}{x_i^2} - 2\cdot \sum_{j\ne i}\frac{2}{(x_i+x_j)^2} + \lambda \prod_{j\ne i}x_j =0\ i=1,\ldots, n;\\
			&\frac{\partial L}{\partial \lambda}=x_1x_2\cdots x_n - 1 = 0.
		\end{align*}
		By symmetry, it is clear that all $x_i$ must be equal. Imposing the constraint $ x_1 x_2 \cdots x_n = 1 $, we obtain the unique solution $ x_1 = x_2 = \cdots = x_n = 1 $. It is straightforward that this critical point is indeed a minimum.
	\end{proof}
	\begin{remark}
		In fact, if $\mu$ is compactly supported and let $X\in(\mathcal{M},\tau)$ such that $X\sim \mu$, the inequality \eqref{ineq:equi LSI} appearing in the proof is equivalent to $\Gamma(X) \geq -\frac{1}{2} \log \left(\Psi(X)\right)$,
		which can be inferred from the Entropy power inequality \eqref{ineq:ent power ineq}.
        Indeed, if we choose $Y=B\sim \mathbb{\mathrm{b}}$ in $(\mathcal{M},\tau)$ that is Boolean independent of $X$, then by the scaling property of entropy $\Gamma$, we have
	\begin{equation*}
		\exp(2\Gamma(X+\sqrt{t}B))\ge \exp(2\Gamma(X)) + t\exp(2\Gamma(B)) = \exp(2\Gamma(X))  + t.
	\end{equation*}
	Hence, one can deduce that
	\begin{equation*}
		\frac{d}{dt}\Big|_{t=0} \exp(2\Gamma(X+\sqrt{t}B)) \ge 1.
	\end{equation*}
	By the de Bruijn equality \eqref{eq:de bruijn}, we know that
	\begin{equation*}
		\frac{d}{dt}\Big|_{t=0} \exp(2\Gamma(X+\sqrt{t}B)) = \exp(2\Gamma(X))\Psi(X).
	\end{equation*}
	So we have just proved that
	\begin{equation}
	    \exp(2\Gamma(X))\Psi(X)\ge 1.
	\end{equation}
        In particular, we deduce that $\Psi(X)<+\infty \Rightarrow \Gamma(X)>-\infty$. In any case, we find the proof of Theorem~\ref{thm:LSI} both elegant and straightforward, and therefore we choose to retain it.
	\end{remark}
	As in the classical and free case, we have also the following equivalent statement of LSI.
	\begin{proposition}
            Let $X\in(\mathcal{M},\tau)$ be a self-adjoint random variable with symmetric law, we have the exponential decay of $\Gamma(\cdot|B)$ along the Ornstein-Uhlenbeck process $X_t$:
            \begin{equation}\label{ineq:exp decay}
                \Gamma(X_t| \mathbb{\mathrm{b}})\le e^{-2t} \Gamma(X| \mathbb{\mathrm{b}}).
			\end{equation}
	\end{proposition}
    Note that when $t=0$, the left and the right hand side of \eqref{ineq:exp decay} are equal. Hence we deduce that 
		\begin{equation*}
			\frac{d}{dt}\Big|_{t=0} \Gamma(X_t| \mathbb{\mathrm{b}})\le \frac{d}{dt}\Big|_{t=0} e^{-2t} \Gamma(X| \mathbb{\mathrm{b}}),
		\end{equation*}
		which just gives back the LSI.
	\begin{proof}
By the de Bruijn identity \eqref{eq:rela fisher}, with the LSI, we have
 \begin{equation*}
			\frac{d}{dt} \Gamma(X_t| \mathbb{\mathrm{b}}) = - \Psi(X_t| \mathbb{\mathrm{b}})\le -2 \Gamma(X_t| \mathbb{\mathrm{b}}).
	\end{equation*}
		Then Gronwall's lemma yields the exponential decay.
	\end{proof}
    We emphasize that the assumption of symmetry in Theorem \ref{thm:LSI} can be removed because the quantity involved only depends on the push forward measure $\mu^{(2)}\in\mathcal{M}(\R_+)$. However, the above equivalences fail under such extension, as the argument relies on the de Bruijn identity. The following result is direct from Theorem \ref{thm:TTI} and Theorem \ref{thm:LSI}.
	\begin{corollary}
	    For any $\mu\in\probmsymr$,we have the bound $W_2^2(\mu,\mathbb{\mathrm{b}})\leq \Psi(\mu| \mathbb{\mathrm{b}}).$
	\end{corollary}

\section{The Non-microstates approach}
\subsection{The Boolean Schwinger-Dyson equation}\label{sec:Schwinger} Let us recall informally the notion of classical and free Fisher information in order to motivate our upcoming definition. In particular, we will not explore the minimal regularity of the test-functions $f$ occurring for the equations to hold. In the classical case, we have the integration by parts formula
$$\int x\cdot f(x)\,d\gamma(x)=\int f'(x)\,d\gamma(x)$$
satisfied by the standard Gaussian distribution $\gamma$. The Fisher information of any measure $\mu$, if it exists, is given as the norm $\|\xi\|_{L^2(\mu)}^2$ of the \textit{score function} $\xi$, a function which allows to write the following deformation of the integration by part formula:
$$\int \xi(x)\cdot f(x)\,d\mu(x)=\int f'(x)\,d\mu(x).$$
Similarly, in the free case, we have the following Schwinger-Dyson formula
$$\int x\cdot f(x)\,d\sigma(x)=\iint \partial f(x,y)\,d\sigma(x)d\sigma(y)$$
satisfied by the standard semicircular distribution $\sigma$, where the $\partial f$ is the non-commutative derivative $\partial f(x,y)=(f(x)-f(y))/(x-y)$. The free Fisher information of any measure~$\mu$, if it exists, is given as $\|\xi\|_{L^2(\mu)}^2$ for the \textit{conjugate variable}, a function which allows to write the following deformation of the Schwinger-Dyson formula:
$$\int \xi(x)\cdot f(x)\,d\mu(x)=\iint \partial f(x,y)\,d\mu(x)d\mu(y).$$
If we want to follow the same path in the Boolean case, the first step is to find a similar formula satisfied by the Rademacher variable $\mathrm{b}=\frac{1}{2}\delta_{-1}+\frac{1}{2}\delta_{+1}$. Let us define $D^\downarrow f(x):=\frac{f(x)-f(0)}{x}$ for functions $f:\mathbb{R}\to \mathbb{R}$ which are differentiable at $0$. For $\mathrm{b}=\frac{1}{2}\delta_{-1}+\frac{1}{2}\delta_{+1}$, we have
$$\int x\cdot f(x)\,d\mathrm{b}(x)=\int D^\downarrow f(x)\,d\mathrm{b}(x)$$
since $x=1/x$ on the support of $\mathrm{b}$.
The operator $D^\downarrow$ will be useful in Section~\ref{sec:Steinmethod} in order to use Stein's method in the Boolean setting.
For symmetric measures $\mu$, we can say that $x\mapsto 1/x$ plays the role of the score function, or of the conjugate variable,  in the Boolean case. Indeed, if $\mu\in \mathcal{M}^{sym}(\mathbb{R})$, for all measurable functions $f$ which are differentiable at $0$, we have
$$\int \frac{1}{x}\cdot f(x)\,d\mu(x)=\int D^\downarrow f(x)\,d\mu(x)$$
whenever both sides are integrable.

\subsection{Boolean entropy and Fisher information}\label{Sec1:booleanentropy*}
Since the function $x \mapsto 1/x$ plays the role of a Boolean score function
in the Boolean Schwinger–Dyson identity, we take $\|1/x\|^2_{L^2(\mu)}$ as a definition for the Non-microstates Boolean Fisher information.

\begin{definition}
Let $X$ be a self-adjoint operator of a non-commutative probability space $(\mathcal{M},\tau)$ with law $\mu\in\mathcal{M}^{sym}(\R)$. The (Non-microstates) \textit{Boolean Fisher information} $\Psi^*(X)$ of $X$ is defined as
\begin{equation}\label{Fisher*}
		\Psi^*(X):=\int \frac{1}{x^2}\,d\mu(x)\in (0,+\infty].
	\end{equation}
    More generally, for any $\mu \in\mathcal{M}^{sym}(\R)$, we define the (Non-microstates) \textit{Boolean Fisher information} $\Psi^*(\mu)$ of $\mu$ as
\begin{equation*}
		\Psi^*(\mu):=\int \frac{1}{x^2}\,d\mu(x)\in (0,+\infty].
	\end{equation*}
\end{definition}

If we want to define the Non-microstates Boolean entropy, we need to integrate this Fisher information $\Psi^*$ along the Boolean heat semigroup: starting from a self-adjoint operator $X$ of a non-commutative probability space $(\mathcal{M},\tau)$ whose law is symmetric, we consider the variable $X+\sqrt{t}B$ where $B\in \mathcal{M}$ is Boolean independent from $X$ and distributed according to the $\mathrm{b}.$ The derivative of $\Psi^*(X+\sqrt{t}B)$ will produce the following quantity (see Theorem~\ref{entfisher*}).

\begin{definition}
Let $X$ be a self-adjoint operator of a non-commutative probability space $(\mathcal{M},\tau)$ with law $\mu\in\mathcal{M}^{sym}(\R)$. The (Non-microstates) \textit{Boolean entropy} $\Gamma^*(X)$ of $X$ is defined as
\begin{equation}\label{Gamma*}
		\Gamma^*(X):=-\frac{1}{2}\log \int \frac{1}{x^2} \,d\mu(x)\in ]-\infty,+\infty].
	\end{equation}
    More generally, for any $\mu \in\mathcal{M}^{sym}(\R)$, we define the (Non-microstates) \textit{Boolean entropy} $\Gamma^*(\mu)$ of $\mu$ as
\begin{equation*}
		\Gamma^*(\mu):=-\frac{1}{2}\log \int \frac{1}{x^2}\,d\mu(x)\in ]-\infty,+\infty].
	\end{equation*}
\end{definition}
In contrast with the Microstates framework, the Boolean heat flow admits a
closed-form evolution for $\Psi^*$; this makes the following de Bruijn identity a
direct computation rather than a variational statement.
	\begin{theorem}\label{entfisher*}
Let $X$ be a self-adjoint operator of a non-commutative probability space $(\mathcal{M},\tau)$, with law $\mu\in\mathcal{M}^{sym}(\R)$,  and $B\in \mathcal{M}$ be Boolean independent from $X$ and distributed according to~$\mathrm{b}$. Fix $t\geq 0$ such that $\Gamma^*(X+\sqrt{s}B)<\infty,\ \Psi^*(X+\sqrt{s}B)<\infty$ for any $s\in[0,t].$ Then we have
		\begin{equation}\label{debruijn*}
		    \Gamma^*(X+\sqrt{t}B) - \Gamma^*(X) = \int_0^t \frac{1}{2}\Psi^*(X + \sqrt{s}B )\,ds.
	    \end{equation}
	\end{theorem}
    \begin{proof}
        By \eqref{eq:Cauchytimet}, we know that
        \begin{align*}
            \Psi^*(X+\sqrt{t}B)&=-G_{(X+\sqrt{t}B)^2}(0)\\
            &=-\frac{G_{X^2}(0)}{1-tG_{X^2}(0)}\\
            &=\frac{\Psi^*(X)}{1+t\Psi^*(X)}.
        \end{align*}
        In particular, we can compute
        \begin{align*}
            \int_0^t\frac{1}{2}\Psi^*(X+\sqrt{s}B)\,ds &=\frac{1}{2}\log(1+t\Psi^*(X))-\frac{1}{2}\log(1)\\
            &=-\frac{1}{2}\log\left(\frac{\Psi^*(X)}{1+t\Psi^*(X)}\right)+\frac{1}{2}\log(\Psi^*(X))\\
            &=\Gamma^*(X+\sqrt{t}B) - \Gamma^*(X).
        \end{align*}
    \end{proof}

From the definitions of $\Gamma^*(X)$ and $\Psi^*(X)$, we have the expected scaling properties.
\begin{proposition}
    Let $X$ be a self-adjoint operator of a non-commutative probability space $(\mathcal{M},\tau)$ with law $\mu\in\mathcal{M}^{sym}(\R)$, and let $\lambda\in \mathbb{R}\setminus \{0\}$.
    
    We have $\Gamma^*(X)=-\frac{1}{2}\log(\Psi^*(X));\ \Psi^*(\lambda X) =\lambda^{-2} \Psi^*(X)$ and $ \Gamma^*(\lambda X) = \Gamma^*(X) + \log |\lambda|$.
\end{proposition}

Following the computations of the previous section for $\Gamma(X)$ and $\Psi(X)$, we also get the following.

	\begin{corollary}\label{Entfisher*}
		We have
		\begin{equation}
			\Gamma^*(X)=\frac{1}{2} \int_0^\infty \left( \frac{1}{s+1}-\Psi^*(X+\sqrt{s}B) \right)\,ds
		\end{equation}
	\end{corollary}
\subsection{Information-theoretic properties}\label{Sec2:booleanentropy*}
In the Non-microstates case, a lot of inequalities of the previous section become equalities.
	\begin{proposition}
	    [Boolean Stam equality]\label{Corbooleanstam*}
	Let $X,Y$ be self-adjoint operators of a non-commutative probability space $(\mathcal{M},\tau)$ with symmetric laws and which are Boolean independent. We have
	\begin{equation}\label{Voistam*}
		\frac{1}{\Psi^*(X+Y)}=\frac{1}{\Psi^*(X)} + \frac{1}{\Psi^*(Y)}.
	\end{equation}
	\end{proposition}
\begin{proof}
For all $z\in \mathbb{H}_+$, we have
$$
       G_{X^2+Y^2}(z)=\frac{1}{z-K_{X^2+Y^2}(z)}
        =\frac{1}{1/G_{X^2}(z)+1/G_{Y^2}(z)-z}
$$
    which extends to $z$ negative real numbers and yields : for $\varepsilon >0$, we have
    $$\tau\left(\frac{1}{-\varepsilon-X^2-Y^2}\right)=\frac{1}{\frac{1}{\tau\left(\frac{1}{-\varepsilon-X^2}\right)}+\frac{1}{\tau\left(\frac{1}{-\varepsilon-Y^2}\right)}+\varepsilon},$$
    or equivalently
    $$\tau\left(\frac{1}{\varepsilon+X^2+Y^2}\right)=\frac{1}{\frac{1}{\tau\left(\frac{1}{\varepsilon+X^2}\right)}+\frac{1}{\tau\left(\frac{1}{\varepsilon+Y^2}\right)}-\varepsilon}.$$
    We conclude using monotone convergence:
    $$\Psi^*(X+Y)=\lim_{\varepsilon\to 0}\tau\left(\frac{1}{\varepsilon+X^2+Y^2}\right),\ \Psi^*(X)=\lim_{\varepsilon\to 0}\tau\left(\frac{1}{\varepsilon+X^2}\right)\ \text{ and }\ \Psi^*(Y)=\lim_{\varepsilon\to 0}\tau\left(\frac{1}{\varepsilon+Y^2}\right).$$
\end{proof}

Thus, many information-theoretic inequalities from the Microstates approach
collapse here into exact identities, reflecting the linear behaviour of
Non-microstates quantities under Boolean convolution. Because $\Gamma^*(X)=-\frac{1}{2}\log \Psi^*(X)$, we get the following.

    	\begin{corollary}[Entropy power equality]
	Let $X,Y$ be self-adjoint operators of a non-commutative probability space $(\mathcal{M},\tau)$ with symmetric laws and which are Boolean independent. We have
	\begin{equation}
	 	\exp(2\Gamma^*(X+Y))=  \exp(2\Gamma^*(X)) + \exp(2\Gamma^*(Y)),
	\end{equation}
	\end{corollary}
    \begin{proof}
From the Boolean Stam equality, we have
$$  \exp(2\Gamma^*(X+Y))=\frac{1}{\Psi^*(X+Y)}=\frac{1}{\Psi^*(X)}+\frac{1}{\Psi^*(Y)}=\exp(2\Gamma^*(X))+\exp(2\Gamma^*(Y)).
$$
    \end{proof}

Except for the proof of the Blachman-Stam inequality and the proof of the Cramér-Rao inequality, all the following properties are proved following verbatim the computations of the previous section for $\Gamma(X)$ and $\Psi(X)$.
	\begin{proposition}\label{prop:Blachmann-Stam*}
	 	The Boolean entropy $\Gamma^*$ and Fisher information $\Psi^*$ satisfy the following properties  (where we assume that all appearing variables are self-adjoint and have symmetric probability measures in some non-commutative probability space):
	 	\begin{enumerate}
	 		\item We have the bound for $\Psi^*$ as follows:
	 		\begin{equation}\label{lowupboudnfisher*}
	 			 \Psi(X)\le \Psi^*(X).
	 		\end{equation}
	 		\item $\Gamma^*$ is upper semicontinuous with respect to the weak topology of probability measures with compact support, and $\Psi^*$ is lower semicontinuous.
	 		\item If $X$ and $Y$ are Boolean independent, then for any $\theta\in [0,1]$ we have the corresponding Shannon-Stam inequality and Blachmann-Stam inequality:
	 		\begin{align}
	 			&\Gamma^*(\sqrt{\theta}X + \sqrt{1-\theta}Y) \ge \theta\Gamma^*(X) + (1-\theta)\Gamma^*(Y)\label{SSI*};\\
	 			&\Psi^*(\sqrt{\theta}X + \sqrt{1-\theta}Y)\le \theta\Psi^*(X) + (1-\theta)\Psi^*(Y)\label{BSI*}.
	 		\end{align}
	 		
	 	\end{enumerate}
	\end{proposition}
	\begin{proof}We just need to justify the proof of the Blachmann-Stam inequality, as the other proofs are mutatis mutandis the same proof as for $\Psi$ and $\Gamma$. We use Voiculescu-Stam equality, the scaling property of $\Psi^*$ and Jensen inequality applied to the inverse function:
    	\begin{equation*}
		\Psi^*(\sqrt{\theta}X+\sqrt{1-\theta}Y)=\frac{1}{\frac{1}{\Psi^*(\sqrt{\theta}X)} + \frac{1}{\Psi^*(\sqrt{1-\theta}Y)}}=\frac{1}{\theta\frac{1}{\Psi^*(X)} + (1-\theta)\frac{1}{\Psi^*(Y)}}\leq \theta \Psi^*(X)+(1-\theta)\Psi^*(Y).
	\end{equation*}
	\end{proof}

	\begin{corollary}[Cramér-Rao inequality]Let $X$ be a self-adjoint operator of a non-commutative probability space $(\mathcal{M},\tau)$ with a symmetric law. We have
	\begin{equation}
		\Psi^*(X)\ge \frac{1}{\tau(X^2)}.
	\end{equation}
	\end{corollary}
	\begin{proof}
Because $\Psi^*(X)\ge \Psi(X)$, it follows from the Cramér-Rao inequality for $\Psi(X)$.
	\end{proof}

We have seen that in contrast with the Microstates framework, the Non-microstates
quantities behave linearly under Boolean convolution. This leads to the following result.

	    \begin{theorem}[Constancy in the Central Limit Theorem]\label{th:constancy}
        Let $\{X_i\}_{i\ge 1}\subset (\mathcal{M},\tau)$ be a sequence of i.i.d.\ (in the Boolean sense) random variables with law $\mu\in\probmsymr $. Define
        \begin{equation*}
            S_n:=\frac{X_1+\cdots+X_n}{\sqrt{n}}.
        \end{equation*}
        Then, for any $n\geq 1$, one has
        \[
        \Gamma^*(S_n)= \Gamma^*(X_1)
        \qquad\text{and}\qquad
        \Psi^*(S_n)= \Psi^*(X_1).
        \]
    \end{theorem}
\begin{proof}
Since $\Gamma^*(X)=-\frac{1}{2}\log(\Psi^*(X))$, we just need to check it for the Fisher information: it is a consequence of Voiculescu-Stam equality combined with the scaling property of $\Psi^*$.
\end{proof}
Overall, the Non-microstates framework exhibits a striking rigidity: many
functional inequalities that are genuinely analytic in the Microstates
setting become explicit algebraic identities or elementary bounds.
\subsection{Functional inequalities}\label{sec:nmfunctional_inequalities}In classical and in free probability, the Fisher information \textit{relative} to the quadratic potential is equal to $\|\xi(x)-x\|_{L^2(\mu)}^2$ where $\xi$ is the score function in the classical case and $\xi$ is the conjugate variable in the free case. Following the same idea, and because $x\mapsto 1/x$ is the score function in the Boolean framework, we define now the Boolean Fisher information relative to the Rademacher variable  as $\|1/x-x\|_{L^2(\mu)}^2=\|1/x\|_{L^2(\mu)}^2+\|x\|_{L^2(\mu)}^2-2$.

        \begin{definition}
let $X$ be a self-adjoint operator of a non-commutative probability space $(\mathcal{M},\tau)$ with law $\mu\in\mathcal{M}^{sym}(\R)$. The (Non-microstates) \textit{Boolean Fisher information $\Psi^*(X|\mathbb{\mathrm{b}})$ of $X$ relative to $\mathbb{\mathrm{b}}$}  is defined as
\begin{equation}
		\Psi^*(X|\mathbb{\mathrm{b}}):=\int \frac{1}{x^2}\,d\mu(x)+ \int x^2\,d\mu(x)-2\in [0,+\infty].
	\end{equation}
    More generally, for any $\mu \in\mathcal{M}^{sym}(\R)$, we define the (Non-microstates) \textit{Boolean Fisher information $\Psi^*(\mu|\mathbb{\mathrm{b}})$ of $\mu$ relative to $\mathbb{\mathrm{b}}$}   as
\begin{equation*}
		\Psi^*(\mu|\mathbb{\mathrm{b}}):=\int \frac{1}{x^2}\,d\mu(x)+ \int x^2\,d\mu(x)-2\in [0,+\infty].
	\end{equation*}
\end{definition}
We would like to define the relative Boolean entropy as in the Microstates case, by $\Gamma^*(X|\mathbb{\mathrm{b}})=-\Gamma^*(X)+\frac{1}{2}\tau\left(X^2\right)-1/2$, so that its derivative along the Ornstein–Uhlenbeck flow equals minus the relative Fisher information. Let us see how it behaves with respect to the Ornstein-Uhlenbeck process. Starting from a self-adjoint operator $X$ of a non-commutative probability space $(\mathcal{M},\tau)$ whose law is symmetric, we consider the variable $X_t:=e^{-t}X + \sqrt{1-e^{-2t}}B$ where $B\in \mathcal{M}$ is Boolean independent from $X$ and distributed according to the Rademacher law $\mathrm{b}.$ We denote by $\mu_{X_t}$ the law of $X_t$.  By the scaling property of $\Gamma$, we know that
	\begin{equation*}
		\Gamma^*(X_t)=\Gamma^*\left(X+\sqrt{e^{2t}-1}B\right)-t.
	\end{equation*}
	Thus, using the scaling property of $\Psi^*$, we have
	\begin{equation}\label{EntOU*}
		\frac{d}{dt}\Gamma^*(X_t)=e^{2t}\Psi^*\left(X+\sqrt{e^{2t}-1}B\right)-1= \Psi^*(X_t)-1.
	\end{equation}We get
	\begin{align}\label{relafisher*}
		\nonumber
		&\frac{d}{dt}\left(-\Gamma^*(X_t)+\frac{1}{2}\tau\left(X_t^2\right)-1/2\right)\\
        &= -\Psi^*(X_t) +1 - e^{-2t}\int x^2 d\mu(x) + e^{-2t}\nonumber \\
		&=  -\Psi^*(X_t|\mathbb{\mathrm{b}}).
	\end{align} This justifies the following definition.
 \begin{definition}
let $X$ be a self-adjoint operator of a non-commutative probability space $(\mathcal{M},\tau)$ with law $\mu\in\mathcal{M}^{sym}(\R)$. The \textit{(Non-microstates) Boolean entropy $\Gamma^*(X|\mathbb{\mathrm{b}})$ of $X$ relative to $\mathbb{\mathrm{b}}$}  is defined as
\begin{equation}\label{entropy*}
		\Gamma^*(X|\mathbb{\mathrm{b}}):=\frac{1}{2}\int x^2 \,d\mu(x)  +\frac{1}{2}\log\int \frac{1}{x^2}\,d\mu(x)-\frac{1}{2}\in [0,+\infty].
	\end{equation}
    More generally, for any $\mu \in\mathcal{M}^{sym}(\R)$, we define the \textit{(Non-microstates) Boolean entropy $\Gamma^*(\mu|\mathbb{\mathrm{b}})$ of $\mu$ relative to $\mathbb{\mathrm{b}}$}   as
\begin{equation*}
		\Gamma^*(\mu|\mathbb{\mathrm{b}}):=\frac{1}{2}\int x^2 \,d\mu(x)  +\frac{1}{2}\log \int \frac{1}{x^2}\,d\mu(x)-\frac{1}{2}\in [0,+\infty].
	\end{equation*}
\end{definition}

Let us compare the relative micro-state entropy with the Non-microstates one, showing the positivity of $\Gamma^*(X|\mathbb{\mathrm{b}})$.
    \begin{proposition}[Microstates  vs Non-microstates]\label{thm:micro-nonmicro}
	For any $\mu\in\mathcal{M}^{sym}(\R)$, we have
	\begin{equation*}
		\Gamma(X|\mathbb{\mathrm{b}})\leq \Gamma^*(X|\mathbb{\mathrm{b}})
\ \ \ \text{ and }\ \ \ 
\Psi(X|\mathbb{\mathrm{b}})\leq \Psi^*(X|\mathbb{\mathrm{b}}).
	\end{equation*}
	\end{proposition}
In particular, $\Gamma^*(\mu|\mathbb{\mathrm{b}})$ provides a universal upper bound
for the Microstates relative entropy, hence positivity of the Non-microstates
functional is automatic once the Microstates one is known.
\begin{proof}
        The first inequality follows from Jensen inequality:
        $$-\int\log x^2\,d\mu(x) =\int\log 1/x^2\,d\mu(x) \leq \log \int 1/x^2\,d\mu(x).$$
For the second one, it is a consequence of $\Psi(X)\leq \Psi^*(X)$.
    \end{proof}
    
We have already proved the logarithmic Sobolev inequality (LSI) in the Microstates
framework. We now show that the same principle holds in the Non-microstates
setting.
\begin{theorem}[LSI]\label{thm:nmLSI}
		For any $\mu\in\mathcal{M}^{sym}(\R)$, we have
	\begin{equation*}
		\Gamma^*(X|\mathbb{\mathrm{b}})\leq \frac{1}{2}\Psi^*(X|\mathbb{\mathrm{b}})
	\end{equation*}
	\end{theorem}
    Note that this inequality is sharp, with equality at the Rademacher law. As in the Microstates case, following the same proofs, a direct consequence is the exponential decay of relative entropy along the Ornstein-Uhlenbeck process $X_t$:
    			\begin{equation*}
				 \Gamma^*(X_t| \mathbb{\mathrm{b}})\le e^{-2t} \Gamma^*(X| \mathbb{\mathrm{b}}).
			\end{equation*}
    \begin{proof}
    In contrast with the Microstates case, the logarithmic Sobolev inequality
here reduces to an elementary convexity inequality.    It is just the inequality $$\frac{1}{2}\log \int \frac{1}{x^2}\,d\mu(x)+\frac{1}{2}\leq\frac{1}{2}  \int \frac{1}{x^2}\,d\mu(x)-\frac{1}{2}-\frac{1}{2}= \frac{1}{2}  \int \frac{1}{x^2}\,d\mu(x)-1.$$
    \end{proof}

Combining the two previous results, we get a log-Sobolev inequality which connects the Microstate entropy $\Gamma$ with the Non-microstate Fisher information $\Psi^*$:
\begin{equation*}\label{LSI:cross}
		\Gamma(X|\mathbb{\mathrm{b}})\leq \frac{1}{2}\Psi^*(X|\mathbb{\mathrm{b}})
	\end{equation*}
One can refine this log-Sobolev inequality by interpolating entropy with transport and Fisher information. This leads to a Boolean HWI inequality, in the spirit of Otto–Villani, from which the log-Sobolev inequality above follows immediately.
In the classical~\cite{otto2000generalization} and free settings~\cite{popescu2013local}, HWI inequalities takes the exact same form (with the corresponding entropy and Fisher information). They play a central role in connecting functional inequalities to transportation estimates.

\begin{theorem}[HWI inequality]\label{thm:HWI_sym}
Let $\mu\in\mathcal{M}^{sym}(\R)$ be a symmetric probability measure with $m_2(\mu)<\infty$.
Then one has the following HWI inequality:
\begin{equation}\label{eq:HWI_sym}
\Gamma(\mu|\mathbb{\mathrm{b}})
\leq
W_2(\mu,\mathbb{\mathrm{b}})\,\sqrt{\Psi^*(\mu|\mathbb{\mathrm{b}})}
-\frac12\,W_2(\mu,\mathbb{\mathrm{b}})^2.
\end{equation}
\end{theorem}
\label{rem:HWI_implies_LSI_star}
By the elementary inequality $ab-\frac12a^2\le \frac12 b^2$, the HWI inequality implies
\[
\Gamma(\mu|\mathbb{\mathrm{b}})\le \frac12\,\Psi^*(\mu|\mathbb{\mathrm{b}}).
\]
\begin{proof}
If $\int x^{-2}\,d\mu(x)=+\infty$, then $\Psi^*(\mu|\mathbb{\mathrm{b}})=+\infty$ and there is nothing to prove. We may therefore assume $\int x^{-2}\,d\mu(x)<+\infty$, which also implies $\Gamma(\mu|\mathbb{\mathrm{b}})<+\infty$ and $\Psi^*(\mu|\mathbb{\mathrm{b}})<+\infty$.

Since $\mu$ is symmetric, the optimal coupling with $\mathbb{\mathrm{b}}$ gives
\[
W_2(\mu,\mathbb{\mathrm{b}})^2=\int\bigl(|x|-1\bigr)^2\,d\mu(x).
\]
The two other quantities appearing in \eqref{eq:HWI_sym} are
\begin{align*}
\Gamma(\mu|\mathbb{\mathrm{b}})
&=\frac12\int x^2\,d\mu(x)-\int\log|x|\,d\mu(x)-\frac12,\\
\Psi^*(\mu|\mathbb{\mathrm{b}})
&=\int\bigl(x^2+x^{-2}-2\bigr)\,d\mu(x).
\end{align*}
For $r>0$ define the potential
\[
V(r):=\frac12 r^2-\log r-\frac12.
\]
Then, setting $r=|x|$,
\[
\Gamma(\mu|\mathbb{\mathrm{b}})=\int V(|x|)\,d\mu(x).
\]
It is straightforward to verify the pointwise inequality
\[
V(r)\le \Bigl(r-\frac1r\Bigr)(r-1)-\frac12(r-1)^2,\qquad r>0,
\]
because
\[
\Bigl(r-\frac1r\Bigr)(r-1)-\frac12(r-1)^2 - V(r) = \log r + \frac1r -1 \ge 0
\]
Integrating with respect to $\mu$,
\[
\Gamma(\mu|\mathbb{\mathrm{b}})
\le
\int\Bigl(|x|-\frac1{|x|}\Bigr)\bigl(|x|-1\bigr)\,d\mu(x)
-\frac12\,W_2(\mu,\mathbb{\mathrm{b}})^2.
\]

By Cauchy--Schwarz,
\[
\int\Bigl(|x|-\frac1{|x|}\Bigr)\bigl(|x|-1\bigr)\,d\mu(x)
\le
\Bigl(\int\bigl(|x|-1\bigr)^2\,d\mu(x)\Bigr)^{1/2}
\Bigl(\int\Bigl(|x|-\frac1{|x|}\Bigr)^2\,d\mu(x)\Bigr)^{1/2}.
\]
The first factor on the right is exactly $W_2(\mu,\mathbb{\mathrm{b}})$. For the second,
\[
\int\Bigl(|x|-\frac1{|x|}\Bigr)^2\,d\mu(x)
=\int\bigl(x^2+x^{-2}-2\bigr)\,d\mu(x)=\Psi^*(\mu|\mathbb{\mathrm{b}}).
\]
Therefore the integral is at most
\[
W_2(\mu,\mathbb{\mathrm{b}})\sqrt{\Psi^*(\mu|\mathbb{\mathrm{b}})}.
\]Substituting back yields the desired HWI inequality \eqref{eq:HWI_sym}.
\end{proof}

\section{Stein's method for Boolean independence}
In this section we do not work inside a non-commutative probability space.
Instead, we develop a Stein-type method for the Rademacher distribution on a classical probability space $(\Omega, \mathcal{A},\mathbb{P})$, which will later be applied to Boolean convolution.
\subsection{Stein's method}\label{sec:Steinmethod}
First of all, let us briefly recall Stein’s method \cite{chen2010normal,goldschmidt2000chen,ross2011fundamentals}, which allows one to bound distances of the form
$$\sup_{\varphi}|\mathbb{E}[\varphi(X)]-\mathbb{E}[\varphi(B)]|$$
for a reference variable $B$. The steps are as follows:
\begin{itemize}\item Find an operator $\mathcal{D}$ such that $\mathbb{E}[\mathcal{D}f(B)]=0$. It is called the \textit{Stein operator} of $B$.
\item Solve the Stein equation
$\mathcal{D}f(x)= \varphi(x)$
for a fixed $\varphi$. Writing $f_\varphi$ the solution of $\mathcal{D}f(x)= \varphi(x)-\mathbb{E}[\varphi(B)]$, we have $$\sup_{\varphi}|\mathbb{E}[\varphi(X)]-\mathbb{E}[\varphi(B)]|=\sup_{\varphi}|\mathbb{E}[\mathcal{D}f_\varphi(X)]|$$
where the supremum is taken over a certain class of functions.
\item It remains to bound the quantities
$|\mathbb{E}[\mathcal{D}f_\varphi(X)]|$.
\end{itemize}
We now specialize Stein’s method to the Rademacher distribution.

$ $

\noindent \textbf{Step 1.}
Let $B$ be a random variable  with Rademacher distribution $\mathbb{\mathrm{b}}$. 
Since $B^2=1$, we have $B=1/B$ almost surely. As a consequence, if we define $D^\downarrow f(x)=\frac{f(x)-f(0)}{x}$ for measurable functions $f:\mathbb{R}\to \mathbb{R}$ which are differentiable at $0$ (with the continuous extension $D^\downarrow f(0)=f'(0)$),
we have
$$\mathbb{E}[B\cdot f(B)]=\mathbb{E}[D^\downarrow f(B)]$$ whenever all the terms are integrable.
We define the Stein operator $\mathcal{D}$ on measurable functions $f:\mathbb{R}\to \mathbb{R}$ which are differentiable at $0$ by
$$\mathcal{D}f(x)=D^\downarrow f(x)-xf(x),$$
in such a way that $\mathbb{E}[\mathcal{D}f(B)]=0$.

$ $

\noindent \textbf{Step 2.} Let us solve the Stein equation in the following proposition.
\begin{proposition}For any measurable function $\varphi:\mathbb{R}\to \mathbb{R}$ which is continuous at $0$ and such that $\varphi(1)=\varphi(-1)=0$, there exists a measurable function $f:\mathbb{R}\to \mathbb{R}$, such that $f(0)=0$ and such that
$$\varphi=\mathcal{D}f.$$ 
Furthermore, whenever $\varphi$ is $1$-Lipschitz, we have
 $\|D^\downarrow f\|_\infty\leq 1.$
\end{proposition}
\begin{proof}It suffices to set $f(1)=f(-1)=0$ and $f(x)=x\frac{\varphi(x)}{1-x^2}$ if $x\neq \pm 1$. The assumption that 
$\varphi$ is continuous at $0$ ensures that 
$f$ is differentiable at 
$0$. The bound on $D^\downarrow f$ is because  $$|D^\downarrow f(x)|=\left|\frac{1}{1+x}\right|\left|\frac{\varphi(x)-\varphi(1)}{x-1}\right|\leq 1$$
if $x\geq 0$ and
$$|D^\downarrow f(x)|=\left|\frac{1}{1-x}\right|\left|\frac{\varphi(x)-\varphi(-1)}{x-(-1)}\right|\leq 1$$
if $x< 0$.
\end{proof}

$ $

\noindent \textbf{Step 3.} We now apply the Stein equation to control 
the $1$-Wasserstein distance $W_1$ between a random variable $X\in L^1$ and $B$. By the Kantorovich duality, we have
$$W_1(X,B)=\sup_{\varphi\ 1-\text{Lipschitz}}|\mathbb{E}[\varphi(X)]-\mathbb{E}[\varphi(B)]|.$$
Asume that $\mathbb{E}[X]=0$. Since $B$ is centered, we can write
$$
\mathbb{E}[\varphi(X)]-\mathbb{E}[\varphi(B)]=\mathbb{E}[\tilde{\varphi}(X)]-\mathbb{E}[\tilde{\varphi}(B)]$$
for any $\tilde{\varphi}$ such that $\varphi-\tilde{\varphi}$ is affine. In particular, we can choose $\varphi$ such that $\varphi(1)=\varphi(-1)=0$ and write
$$W_1(X,B)=\sup_{\substack{\varphi\ 1-\text{Lipschitz}\\ \varphi(1)=\varphi(-1)=0}}|\mathbb{E}[\varphi(X)]-\mathbb{E}[\varphi(B)]|.$$
Using the solution of the Stein equation yields the following lemma.
\begin{lemma}For any variable $X\in L^1$ such that $\mathbb{E}[X]=0$, we have
$$W_1(X,B)\leq \sup_{f}|\mathbb{E}[\mathcal{D}f(X)]|$$
where the supremum is taken over measurable functions $f$ which are differentiable at $0$, such that $f(0)=0$ and such that $\|D^\downarrow f\|_\infty\leq 1.$
\end{lemma}
We obtain the following bound on the Wasserstein distance.
\begin{proposition}For any variable $X\in L^2$, we have
$$W_1(X,B)\leq |\mathbb{E}[X]|+\mathbb{E}\left[\left|(X-\mathbb{E}[X])^2-1\right|\right].$$
In particular, if $\mathbb{E}[X]=0$,
$$W_1(X,B)\leq \mathbb{E}\left[\left|X^2-1\right|\right].$$
\end{proposition}
\begin{proof}
Because $W_1(X,B)\leq W_1(X,X-\mathbb{E}[X])+W_1(X-\mathbb{E}[X],B)=|\mathbb{E}[X]|+W_1(X-\mathbb{E}[X],B)$, it suffices to prove the inequality in the case where $X$ is centered.

Whenever $\mathbb{E}[X]=0$, we have
$$W_1(X,B)\leq \sup_{f}|\mathbb{E}[\mathcal{D}f(X)]|$$where the supremum is taken over measurable functions $f$ which are differentiable at $0$, such that $f(0)=0$ and such that $\|D^\downarrow f\|_\infty\leq 1.$ For such a function $f$, we have  \[
\mathcal D f(X)=(1-X^2)\,D^\downarrow f(X),
\] which allows to write
$$|\mathbb{E}[\mathcal{D}f(X)]|=|\mathbb{E}[(1-X^2)\cdot D^\downarrow f(X)]|\leq \|1-X^2\|_{L^1}\cdot \|D^\downarrow f\|_\infty\leq \mathbb{E}\left[\left|X^2-1\right|\right]$$
which implies that $W_1(X,B)\leq \mathbb{E}\left[\left|X^2-1\right|\right]$.
\end{proof}
We may now restate the previous bound at the level of laws, and get a quantitative version of the Boolean fourth moment theorem in \cite[Lemma 3.1]{Arizmendi2013}.
\begin{corollary}[Quantitative fourth moment theorem]Let $\mu\in \mathcal{M}(\mathbb{R})$ such that $m_1(\mu) = 0$, $m_2(\mu) = 1$, and $m_4(\mu) < +\infty$. We have
$$W_1(\mu,\mathbb{\mathrm{b}})\leq \sqrt{m_4(\mu)-1}.$$
\end{corollary}
\begin{proof}Let $B$ be a random variable  with Rademacher distribution $\mathbb{\mathrm{b}}$ and $X$ be a random variable  with distribution $\mu$. We have
$$W_1(\mu,\mathbb{\mathrm{b}})=W_1(X,B)\leq \mathbb{E}\left[\left|X^2-1\right|\right]\leq \sqrt{\mathbb{E}\left[\left(X^2-1\right)^2\right]}=\sqrt{\mathbb{E}[X^4]-1}.$$
\end{proof}
Using \eqref{eq:fourthcumulant}, we obtain the following Berry-Esseen bound.
\begin{corollary}[Berry-Esseen bound]Let $\mu\in \mathcal{M}(\mathbb{R})$ such that $m_1(\mu) = 0$, $m_2(\mu) = 1$,
and $m_4(\mu) < +\infty$.  We define $\mu_n$ to be the dilation of $\mu^{\uplus n}$ by $\frac{1}{\sqrt{n}}$, i.e. $\mu_n(B)=\mu^{\uplus n}(\{x:\frac{1}{\sqrt{n}}x\in B\})$. We have
$$W_1(\mu_n,\mathbb{\mathrm{b}})\leq \frac{1}{\sqrt{n}}\sqrt{m_4(\mu)-1}.$$
\end{corollary}

\subsection{Stein discrepancy}\label{sec:Stein_discrepancy}The notion of Stein kernel plays a central role in classical and free probability. We now introduce its Boolean analogue. 
In the classical setting, a Stein kernel~\cite{ledoux2015stein} with respect to the Gaussian distribution is a function $A\in L^2(\mu)$ such that
$$\int x\cdot f(x)d\mu(x)=\int A(x)\cdot f'(x)d\mu(x).$$
In other words, the score function is obtained by replacing $x$ by $\xi$ on the left side in the integration by parts (see Section~\ref{sec:Schwinger}), and the Stein kernel is obtained by replacing $1$ by $A(x)$ on the right side of the equality. Similarly, a free Stein kernel~\cite{fathi2017free} with respect to the semicircular distribution is any $A\in L^2(\mu\otimes \mu)$ such that the following deformation of the Schwinger-Dyson formula holds:
$$\int x\cdot  f(x)\,d\mu(x)=\iint  A(x,y)\cdot \partial f(x,y)\,d\mu(x)d\mu(y).$$
Following the same philosophy, it is natural to interpret the map $x\to x^2$ as a Boolean Stein kernel (with respect to $\mathbb{\mathrm{b}}$) of any centered measure $\mu$,. Indeed,  for any function which is differentiable at $0$, we have the deformation of the Boolean Schwinger-Dyson identity
$$\int x\cdot f(x)\,d\mu(x)=\int x^2\cdot D^\downarrow f (x)\,d\mu(x).$$
This leads to the definition of the Boolean Stein discrepancy, in parallel to the classical Stein discrepancy $\|T-1\|_{L^2(\mu)}$ (see also the free one).
\begin{definition}
Let $X$ be a real random variable with law $\mu\in\mathcal{M}(\R)$. The \textit{Boolean Stein discrepancy} $D^*(X|\mathbb{\mathrm{b}})$ of $X$ with respect to $\mathbb{\mathrm{b}}$ is defined as
\begin{equation*}
		D^*(X|\mathbb{\mathrm{b}}):=\left(\int (x^2-1)^2 \,d\mu(x)\right)^{1/2}\in [0,+\infty].
	\end{equation*}
Similarly, for any $\mu \in\mathcal{M}(\R)$, we define the \textit{Boolean Stein discrepancy} $D^*(\mu|\mathbb{\mathrm{b}})$ of $\mu$ with respect to $\mathbb{\mathrm{b}}$ as
\begin{equation*}
		D^*(\mu|\mathbb{\mathrm{b}}):=\left(\int (x^2-1)^2\,d\mu(x)\right)^{1/2}\in [0,+\infty].
	\end{equation*}
\end{definition}    

The Stein discrepancy has already shown up in the previous Section~\ref{sec:Steinmethod}, where it was shown that whenever $\mu$ is centered, we have
$$W_1(\mu,\mathbb{\mathrm{b}})\leq  D^*(\mu|\mathbb{\mathrm{b}}).$$
We can strengthen this result to the 2-Wasserstein distance as follows. It is the Boolean analogue of the classical and free WS inequalities in respectively \cite{ledoux2015stein,cebron2021quantitative}.
\begin{proposition}[WS inequality]\label{prop:WS}
    For any measure $\mu\in \mathcal{M}(\mathbb{R})$, we have
$$\left(\int (|x|-1)^2\,d\mu(x)\right)^{1/2}\leq  D^*(\mu|\mathbb{\mathrm{b}}).$$
In particular, if $\mu$ is symmetric,
$$W_2(\mu,\mathbb{\mathrm{b}})\leq  D^*(\mu|\mathbb{\mathrm{b}}).$$
\end{proposition}
\begin{proof}
For all $x\in \mathbb{R}$, we have
$$(|x|-1)^2\leq (|x|-1)^2(|x|+1)^2=(|x|^2-1)^2.$$
Therefore,
$$\left(\int (|x|-1)^2\,d\mu(x)\right)^{1/2}\leq \left(\int (|x|^2-1)^2\,d\mu(x)\right)^{1/2}=D^*(\mu|\mathbb{\mathrm{b}}).$$
For $\mu$ symmetric, the equality $W_2(\mu,\mathbb{\mathrm{b}})=\left(\int (|x|-1)^2d\mu(x)\right)^{1/2}$
was established in previous sections.
\end{proof}
As a consequence, inequalities involving the Boolean Stein discrepancy can be turned into Berry-Esseen results using \eqref{eq:fourthcumulant}.
\begin{proposition}[Berry-Esseen bound]\label{prop:Berryesseen}Let $\mu\in \mathcal{M}^{sym}(\mathbb{R})$ such that $m_1(\mu) = 0$, $m_2(\mu) = 1$,
and $m_4(\mu) < +\infty$.  We define $\mu_n$ to be the dilation of $\mu^{\uplus n}$ by $\frac{1}{\sqrt{n}}$, i.e. $\mu_n(B)=\mu^{\uplus n}(\{x:\frac{1}{\sqrt{n}}x\in B\})$. We have
    $$W_2(\mu_n,\mathbb{\mathrm{b}})\leq D^*(\mu_n|\mathbb{\mathrm{b}})=\frac{1}{\sqrt{n}}\sqrt{m_4(\mu)-1}.$$
\end{proposition}
\begin{proof}
It suffices to observe that $m_2(\mu_n)=1$, so we can compute
\begin{align*}
  D^*(\mu_n|\mathbb{\mathrm{b}})^2&= \int (|x|^2-1)^2d\mu_n(x)\\
   &=m_4(\mu_n)-2m_2(\mu_n)+1\\
   &=r_4(\mu_n)\\
   &=\frac{1}{n}r_4(\mu)\\
   &=\frac{1}{n}(m_4(\mu)-1),
\end{align*}
where we used \eqref{eq:fourthcumulant}.
\end{proof}
\subsection{Entropic CLT}\label{sec:EntropicCLT} To turn the convergence of the Stein discrepancy into entropic convergence in the Boolean central limit theorem, we establish a Boolean HSI inequality. 
This inequality is the exact analogue, at the level of its functional form, of the HSI inequalities in the classical and free settings~\cite{ledoux2015stein,fathi2017free}.
\begin{theorem}[HSI inequality]\label{thm:HSI}For any measure $\mu\in \mathcal{M}^{sym}(\mathbb{R})$, we have
$$\Gamma(\mu|\mathbb{\mathrm{b}})\leq \frac{1}{2}D^*(\mu|\mathbb{\mathrm{b}})^2\cdot \log\left(1+\frac{\Psi^*(\mu|\mathbb{\mathrm{b}})}{D^*(\mu|\mathbb{\mathrm{b}})^2}\right).$$
\end{theorem}
Using $\log(1+x)\leq x$, we recover the inequality $2\Gamma(\mu|\mathbb{\mathrm{b}})\leq \Psi^*(\mu|b)$ which is the log-Sobolev inequality between $\Gamma$ and $\Psi^*$. 

\begin{proof}
The proof relies on the pointwise inequality
\[
u - \log u - 1 \leq \delta\!\left(u+\frac1u-2\right) + (-\log\delta+\delta-1)(u-1)^2
\]
for all $u>0$ and $\delta\in(0,1]$, and then integrate and optimize over $\delta$.

Let
\[
F(u,\delta)
= \log u - u + 1
+ \delta\!\left(u+\frac1u-2\right)
+ (-\log\delta+\delta-1)(u-1)^2.
\]
Note that $F\geq 0$ is equivalent to the desired inequality.
Fix $u>0$ and minimize with respect to $\delta\in(0,1]$. Differentiating gives
\[
\partial_\delta F(u,\delta)
= \left(u+\frac1u-2\right) + \left(-\frac1\delta+1\right)(u-1)^2.
\]
Using $u+1/u-2=(u-1)^2/u$, this becomes
\[
\partial_\delta F(u,\delta)
= (u-1)^2 \left(\frac1u - \frac1\delta +1\right).
\]
For $u\neq 1$ the critical point is $\delta_*(u)=u/(u+1)\in(0,1)$. The second derivative
\[
\partial_\delta^2 F(u,\delta) = (u-1)^2/\delta^2 >0
\]
shows it is a minimum. For $u=1$, $F(1,\delta)=0$ for all $\delta$. Substituting $\delta=u/(u+1)$ yields
\[
\min_{\delta\in(0,1]} F(u,\delta)
= H(u) := \log u - u + 1 + (u-1)^2 \log\!\frac{u+1}{u}.
\]

It remains to show $H(u)\ge 0$ for all $u>0$. Clearly $H(1)=0$. A direct computation gives
\[
H'(u) = 2(u-1)\left( \log\frac{u+1}{u} - \frac1{u+1} \right).
\]
The function $\phi(u)=\log((u+1)/u) - 1/(u+1)$ satisfies $\phi(u)>0$ for all $u>0$ (since $\phi(1)>0$, $\phi(u)\to 0^+$ as $u\to\infty$, and $\phi'(u)<0$). Thus $H'(u)<0$ on $(0,1)$ and $H'(u)>0$ on $(1,\infty)$, so $H$ attains its global minimum at $u=1$, where $H(1)=0$. Hence $H(u)\ge 0$.

Consequently $F(u,\delta)\ge 0$ for all $u>0$ and $\delta\in(0,1]$, which rearranges to the desired pointwise inequality. Integrating with respect to $\mu$ gives
\[
2\Gamma(\mu|\mathbb{\mathrm{b}})
\leq \delta \Psi^*(\mu|\mathbb{\mathrm{b}}) + (-\log\delta + \delta -1) D^*(\mu|\mathbb{\mathrm{b}})^2.
\]

The right-hand side, seen as a function of $\delta$, is minimized at
\[
\delta = \frac{D^*(\mu|\mathbb{\mathrm{b}})^2}{D^*(\mu|\mathbb{\mathrm{b}})^2 + \Psi^*(\mu|\mathbb{\mathrm{b}})}.
\]
Substituting this optimal value yields exactly
\[
2\Gamma(\mu|\mathbb{\mathrm{b}})
\leq D^*(\mu|\mathbb{\mathrm{b}})^2 \cdot \log\left(1+\frac{\Psi^*(\mu|\mathbb{\mathrm{b}})}{D^*(\mu|\mathbb{\mathrm{b}})^2}\right),
\]
which concludes the proof.
\end{proof}
In particular, the Boolean central limit theorem holds at the entropic level, from the convergence of the Stein discrepancy via the HSI inequality.
\begin{corollary}
      Let $\mu\in \mathcal{M}^{sym}(\mathbb{R})$ such that $m_1(\mu) = 0$, $m_2(\mu) = 1$,
and $m_4(\mu) < +\infty$.  We define $\mu_n$ to be the dilation of $\mu^{\uplus n}$ by $\frac{1}{\sqrt{n}}$, i.e. $\mu_n(B)=\mu^{\uplus n}(\{x:\frac{1}{\sqrt{n}}x\in B\})$. We have, as $n\to \infty$,
    $$\Gamma(\mu_n|\mathbb{\mathrm{b}})=O\left(\frac{\log(n)}{n}\right).$$
\end{corollary}
\begin{proof}By Theorem~\ref{th:constancy}, the quantity $\Psi^*(\mu_n|\mathbb{\mathrm{b}})$ is constant and Proposition~\ref{prop:Berryesseen} gives the rate $D^*(\mu_n|\mathbb{\mathrm{b}})^2=O(1/n)$. Thus
$$D^*(\mu_n|\mathbb{\mathrm{b}})^2\cdot \log\left(1+\frac{\Psi^*(\mu_n|\mathbb{\mathrm{b}})}{D^*(\mu_n|\mathbb{\mathrm{b}})^2}\right)=O\left(\frac{\log(n)}{n}\right),$$
and the HSI inequality allows to conclude.
\end{proof}
We can also derive a quantitative bound on the Boolean Fisher information along the Boolean central limit theorem. We do not expect the rate below to be optimal.

\begin{proposition}
Let $\mu\in \mathcal{M}(\mathbb{R})$ such that $m_1(\mu)=0$, $m_2(\mu)=1$ and $m_4(\mu)<+\infty$.  
Let $\mu_n$ be the dilation of $\mu^{\uplus n}$ by $1/\sqrt{n}$, i, i.e. $\mu_n(B)=\mu^{\uplus n}(\{x:\frac{1}{\sqrt{n}}x\in B\})$.
Then
\[
\Psi(\mu_n|\mathbb{\mathrm b})
\le \Psi^*(\mu_n)\,D^*(\mu_n|\mathbb{\mathrm b})
=\frac{1}{\sqrt n}\,m_{-2}(\mu)\sqrt{m_4(\mu)-1}.
\]
\end{proposition}

\begin{proof}
We first prove the general bound
\[
|\Psi(\mu)-1|
\le \Psi^*(\mu)\,D^*(\mu|\mathbb{\mathrm b}).
\]
Using~\eqref{eq:fisher_0^1}, we write
\[
\Psi(\mu)
=\iiint \frac{1}{t x^2+(1-t)y^2}\,dt\,d\mu(x)d\mu(y).
\]
Hence
\[
|\Psi(\mu)-1|
=\left|
\iiint 
\frac{1-(t x^2+(1-t)y^2)}{t x^2+(1-t)y^2}
\,dt\,d\mu(x)d\mu(y)
\right|.
\]

Applying Cauchy--Schwarz with respect to the product measure
$dt\,d\mu(x)d\mu(y)$ gives
\[
|\Psi(\mu)-1|
\le
\left(
\iiint \frac{1}{(t x^2+(1-t)y^2)^2}
\right)^{1/2}
\left(
\iiint (1-(t x^2+(1-t)y^2))^2
\right)^{1/2}.
\]

For the second factor,
\[
\iiint(1-W_t)^2\,dt\,d\mu\,d\mu = \int_0^1\mathrm{Var}(W_t)\,dt,
\]
where $W_t=tX^2+(1-t)Y^2$. Since $\mathrm{Var}(W_t)=[t^2+(1-t)^2]D^*(\mu|\mathrm{b})^2$ and $t^2+(1-t)^2\leq 1$, we obtain
\[
\int_0^1\mathrm{Var}(W_t)\,dt \leq D^*(\mu|\mathrm{b})^2,
\]
so the second factor is at most $D^*(\mu|\mathrm{b})$.

The first factor equals $\Psi^*(\mu)$, because
\[
\iiint \frac{1}{(t x^2+(1-t)y^2)^2}
=
\iint \frac{1}{x^2-y^2}\!\left(\frac{-1}{x^2}-\frac{-1}{y^2}\right)
d\mu(x)d\mu(y)
=\Psi^*(\mu)^2.
\]
Thus
\[
|\Psi(\mu)-1|
\le \Psi^*(\mu)\,D^*(\mu|\mathbb{\mathrm b}).
\]

Since
\[
\Psi(\mu|\mathbb{\mathrm b})
=\Psi(\mu)+m_2(\mu)-2,
\]
we obtain
\[
\Psi(\mu|\mathbb{\mathrm b})
\le \Psi^*(\mu)\,D^*(\mu|\mathbb{\mathrm b})
+|m_2(\mu)-1|.
\]
Applying this to $\mu_n$, we have $m_2(\mu_n)=1$, while
$\Psi^*(\mu_n)=\Psi^*(\mu)$ by Theorem~\ref{th:constancy}
and
$D^*(\mu_n|\mathbb{\mathrm b})=\frac{1}{\sqrt n}\sqrt{m_4(\mu)-1}$
by Proposition~\ref{prop:Berryesseen}.
This yields the claimed bound.
\end{proof}

Beyond its role in the entropic central limit theorem, the HSI inequality above fits into a broader family of interpolation inequalities relating transport, entropy and Fisher information. 
In the classical and free settings, the HSI inequality implies a WSH inequality~\cite{ledoux2015stein,cebron2021quantitative}, which refines the Talagrand inequality. 
We now establish the corresponding Boolean WSH inequality, thereby completing the parallel between the classical, free and Boolean frameworks.
\begin{theorem}[WSH inequality]\label{prop:WSH}
Let $\mu\in\mathcal{M}^{sym}(\R)$ with $m_2(\mu)<+\infty$.
Then
\begin{equation}\label{eq:WSH}
W_2(\mu,\mathbb{\mathrm{b}})
\le
D^*(\mu|\mathbb{\mathrm{b}})\,
\arccos\!\left(
\exp\!\left\{-\frac{\Gamma(\mu|\mathbb{\mathrm{b}})}{D^*(\mu|\mathbb{\mathrm{b}})^2}\right\}
\right).
\end{equation}
\end{theorem}
Observe that for $x\ge 0$ one has
$\arccos(e^{-x}) \le \sqrt{2x}.$
Applying this with 
$x=\Gamma(\mu|\mathbb{\mathrm b})/D^*(\mu|\mathbb{\mathrm b})^2$ 
in~\eqref{eq:WSH}, we obtain the Boolean Talagrand inequality
$
W_2(\mu,\mathbb{\mathrm b})^2
\le
2\,\Gamma(\mu|\mathbb{\mathrm b}).$

\begin{proof}
Let $X$ be a random variable whose law is $\mu$ and, for $t\ge 0$, let $X_t$ be a symmetric random variable such that
\[
X_t^2 = e^{-2t}X^2+(1-e^{-2t}).
\]
We warn the reader that the semigroup $(X_t)_t$ is not the one given by the Boolean Ornstein-Uhlenbeck semi-group as considered in the previous sections.
However,  setting $Y_t=X_t^2$, we have $Y_t\to 1$ as $t\to+\infty$, hence
$X_t\to B$ in distribution.

Since $X_t$ is symmetric,
\[
W_2(\mu_t,\mathbb{\mathrm{b}})^2
=\E\big[(|X_t|-1)^2\big].
\]
Differentiating with respect to $t$ gives
\[
\frac{d}{dt}W_2(\mu_t,\mathbb{\mathrm{b}})^2
=
2\,\E\!\left[(|X_t|-1)\frac{d}{dt}|X_t|\right].
\]
Because $|x|=\sqrt{x^2}$ and $X_t^2=e^{-2t}X^2+1-e^{-2t}$,
\[
\frac{d}{dt}|X_t|
=
\frac{1}{2|X_t|}\frac{d}{dt}X_t^2
=
-\frac{X_t^2-1}{|X_t|}.
\]
Hence
\[
\frac{d}{dt}W_2(\mu_t,\mathbb{\mathrm{b}})^2
=
-2\,\E\!\left[(|X_t|-1)\frac{X_t^2-1}{|X_t|}\right].
\]
Since
\[
\frac{X_t^2-1}{|X_t|}
=
|X_t|-\frac1{|X_t|},
\]
we obtain
\[
\frac{d}{dt}W_2(\mu_t,\mathbb{\mathrm{b}})^2
=
-2\,\E\!\left[(|X_t|-1)\left(|X_t|-\frac1{|X_t|}\right)\right].
\]
By Cauchy--Schwarz,
\[
-\frac{d}{dt}W_2(\mu_t,\mathbb{\mathrm{b}})
\le
\left(
\E\!\left[\left(|X_t|-\frac1{|X_t|}\right)^2\right]
\right)^{1/2}.
\]
Since
\[
\E\!\left[\left(|X_t|-\frac1{|X_t|}\right)^2\right]
=\Psi^*(\mu_t|\mathbb{\mathrm{b}}),
\]
this yields
\begin{equation}\label{eq:WI_bool}
\frac{d^+}{dt}W_2(\mu_t,\mathbb{\mathrm{b}})
\le \sqrt{\Psi^*(\mu_t|\mathbb{\mathrm{b}})}.
\end{equation}

We now compute the evolution of the entropy along the same flow.  
By definition,
\[
\frac{d}{dt}\Gamma(\mu_t|\mathbb{\mathrm{b}})
=
\frac12\,\E\!\left[
\frac{d}{dt}Y_t
-
\frac{1}{Y_t}\frac{d}{dt}Y_t
\right]
=
\frac12\,\E\!\left[
\left(1-\frac1{Y_t}\right)\frac{d}{dt}Y_t
\right].
\]
Since $\frac{d}{dt}Y_t=-2(Y_t-1)$,
\[
\frac{d}{dt}\Gamma(\mu_t|\mathbb{\mathrm{b}})
=
-\,\E\!\left[
\left(1-\frac1{Y_t}\right)(Y_t-1)
\right].
\]
But
\[
\left(1-\frac1{Y_t}\right)(Y_t-1)
=
\frac{(Y_t-1)^2}{Y_t}
=
Y_t+\frac1{Y_t}-2.
\]
Therefore
\[
-\frac{d}{dt}\Gamma(\mu_t|\mathbb{\mathrm{b}})
=
\E\!\left[
Y_t+\frac1{Y_t}-2
\right]
=
\E\!\left[
X_t^2+\frac1{X_t^2}-2
\right]
=
\Psi^*(\mu_t|\mathbb{\mathrm{b}}).
\]

Applying Theorem~\ref{thm:HSI} to $\mu_t$ and using
$D^*(\mu_t|\mathbb{\mathrm{b}})\le D^*(\mu|\mathbb{\mathrm{b}})$, we obtain
\[
\sqrt{\Psi^*(\mu_t|\mathbb{\mathrm{b}})}
\le
-\frac{d}{dt}
\left[
D^*(\mu|\mathbb{\mathrm{b}})
\arccos\!\left(
\exp\!\left\{-\frac{\Gamma(\mu_t|\mathbb{\mathrm{b}})}{D^*(\mu|\mathbb{\mathrm{b}})^2}\right\}
\right)
\right].
\]
Combining this with \eqref{eq:WI_bool} and integrating on $[0,+\infty)$
gives \eqref{eq:WSH}, since $\mu_t\to\mathbb{\mathrm{b}}$ and
$\Gamma(\mathbb{\mathrm{b}}|\mathbb{\mathrm{b}})=0$.
\end{proof}
\section*{Conclusion and outlook}
In this paper, we developed a Boolean information-theoretic framework paralleling the classical and free settings, by introducing Microstates and Non-microstates Boolean Fisher informations and relating them to Boolean entropy through de Bruijn-type identities. We established Boolean analogues of the main functional inequalities, including logarithmic Sobolev, HWI, HSI, Talagrand and WSH-type inequalities, and complemented this approach with a Boolean Stein method, yielding quantitative Berry–Esseen bounds and entropic convergence rates. Taken together, these results complete the transport–entropy–information picture in the Boolean setting and highlight its analogies with  the classical and free frameworks. We also note that the recent work~\cite{garza2026finite} develops an information-theoretic framework for another branch of non-commutative probability, namely finite free probability.

Several natural questions remain open. Wande restricted ourselves to symmetric measures, and we postpone both the non-symmetric theory and multivariate extensions to future work. Moreover, we expect that a genuine HWI-type inequality relating the Microstates Boolean entropy $\Gamma(\mu|\mathbb{\mathrm{b}})$ and the Microstates Boolean Fisher $\Psi(\mu|\mathbb{\mathrm{b}})$ information should hold. These directions will be investigated in subsequent works.

\bibliographystyle{plain}
\bibliography{bibliography.bib}
\end{document}